%% file: quasifuchsian.tex
\documentclass[11pt]{amsart}

\usepackage{amsfonts, amstext, amsmath, amsthm, amscd, amssymb}
\usepackage{epsfig, graphics, psfrag, enumerate}
\usepackage{color}

\let\oldmarginpar\marginpar
\renewcommand\marginpar[1]{\oldmarginpar[\raggedleft\footnotesize #1]%
{\raggedright\footnotesize #1}}

\renewcommand{\setminus}{{\smallsetminus}}

\newcommand{\HH}{{\mathbb{H}}}

\newcommand{\CC}{{\mathbb{C}}}

\newcommand{\bdy}{{\partial}}

\newcommand{\cut}{{\backslash \backslash}}

\newcommand{\abs}[1]{{\left\vert #1 \right\vert}}

\newcommand{\GA}{{\mathbb{G}_A}}

\newcommand{\GRA}{{\mathbb{G}'_A}}
\newcommand{\GRB}{{\mathbb{G}'_B}}

\def\co{\colon\thinspace}

\theoremstyle{plain}
\newtheorem{theorem}{Theorem}[section]

\newtheorem{lemma}[theorem]{Lemma}
\newtheorem{prop}[theorem]{Proposition}

\newtheorem*{namedtheorem}{\theoremname}
\newcommand{\theoremname}{testing}

\theoremstyle{definition}
\newtheorem{define}[theorem]{Definition}
\newtheorem{remark}[theorem]{Remark}
\newtheorem{convention}[theorem]{Convention}

\begin{document}
\title{Quasifuchsian state surfaces}
\author[D. Futer]{David Futer}
\author[E. Kalfagianni]{Efstratia Kalfagianni}
\author[J. Purcell]{Jessica S. Purcell}

\address[]{Department of Mathematics, Temple University,
Philadelphia, PA 19122, USA}

\email[]{dfuter@temple.edu}

\address[]{Department of Mathematics, Michigan State University, East
Lansing, MI 48824, USA}

\email[]{kalfagia@math.msu.edu}

\address[]{ Department of Mathematics, Brigham Young University,
Provo, UT 84602, USA}

\email[]{jpurcell@math.byu.edu }
\thanks{{D.F. is supported in part by NSF grant DMS--1007221.}}

\thanks{{E.K. is supported in part by NSF grant DMS--1105843.}}

\thanks{{J.P. is supported in part by NSF grant  DMS--1007437 and a Sloan Research Fellowship.}}

\thanks{ \today}

\begin{abstract}
This paper continues our study, initiated in \cite{fkp:gutsjp}, of
essential state surfaces in link complements that satisfy a mild
diagrammatic hypothesis (homogeneously adequate).  For hyperbolic links,
we show that the geometric type of these surfaces in the Thurston
trichotomy is completely determined by a simple graph--theoretic
criterion in terms of a certain spine of the surfaces. For links with
$A$-- or $B$--adequate diagrams, the geometric type of the surface is
also completely determined by a coefficient of the colored Jones
polynomial of the link.
\end{abstract}
\maketitle

\input{intro}

\input{annuli}

\input{poly}

\input{links}

\bibliographystyle{hamsplain} \bibliography{biblio}

\end{document}

%% file: intro.tex
\section{Introduction}\label{sec:intro}

A major goal in modern knot theory is to relate the geometry of a knot
complement to combinatorial invariants that are easy to read off  a diagram of the knot.  In a
recent monograph \cite{fkp:gutsjp}, we find connections between
geometric invariants of a knot or link complement, combinatorial properties of
its diagram, and stable coefficients of its colored Jones polynomials.
The bridge among these different invariants consists of \emph{state
  surfaces} associated to Kauffman states of a link diagram
\cite{KaufJones}.  These surfaces lie in the link complement and are
naturally constructed from a diagram, while certain graphs that form a
spine for these surfaces aid in the computation of Jones polynomials
\cite{dasbach-futer...}.

In this paper, we continue the study of these state surfaces, with the
goal of obtaining additional geometric information on a link
complement, and relating it back to diagrammatical and quantum
invariants of the link.  In particular, we establish combinatorial
criteria that characterize the geometric types of state surfaces
in the Thurston trichotomy.  This trichotomy, proved by Thurston
\cite{thurston:notes} and Bonahon \cite{bonahonends}, asserts that
every essential surface in a hyperbolic 3-manifold fits into exactly
one of three types: semi-fiber, quasifuchsian, or accidental. (See
Definition \ref{def:types} below for details.)  We show that under a
mild diagrammatic hypothesis, certain state surfaces will never be
accidental, and a simple graph--theoretic property determines
whether the state surface is a semi-fiber or quasifuchsian. For the class of $A$-- or $B$--adequate diagrams,
which arise in the study of knot polynomial invariants \cite{lick-thistle,
  thi:adequate}, the geometric type of the surface 
is determined by a single coefficient of the colored Jones polynomials of the knot.

The problem of determining the geometric types of essential surfaces in
knot and link complements has been studied fairly well in the literature. For example, Menasco and Reid proved that no alternating link
complement contains an embedded quasifuchsian closed surface
\cite{MenascoReid}, which led to the result that there are no embedded
totally geodesic surfaces in alternating link complements.  More
recently, Masters and Zhang found closed,  immersed
quasifuchsian surfaces in any hyperbolic link complement
\cite{MastersZhang}.

Turning to surfaces with boundary, it is known that all three geometric types occur in hyperbolic link complements.
For example, 
Tsutsumi constructed hyperbolic knots with  accidental {S}eifert surfaces of
arbitrarily high genus  \cite{tsutsumi}. On the other hand, Fenley proved that minimal genus
Seifert surfaces cannot be accidental \cite{fenley:qf-seifert}.  An
alternate proof of this was given by Cooper and Long
\cite{cooper-long}.
Adams showed that checkerboard
surfaces in alternating link complements are quasifuchsian
\cite{adams:quasi-fuchsian}.  Here we give an alternate
proof of this fact, and provide broad  families of non-accidental surfaces  constructed from non-alternating diagrams.

The results of this paper have some direct consequences in hyperbolic geometry.
 First, they dovetail with recent
work of Thistlethwaite and Tsvietkova, who gave an algorithm
to construct the hyperbolic structure on a link complement directly
from a diagram \cite{thist-tsviet, tsvietthesis}.  Their algorithm
works whenever a link diagram admits a non-accidental state surface, which is exactly what our results ensure
for a very large class of
diagrams.  Second, the quasifuchsian surfaces that we construct fit into the machinery developed by Adams
\cite{adams:quasi-fuchsian}.  He showed that if a cusped hyperbolic
manifold contains a properly embedded quasifuchsian surface with
boundary, then there are restrictions on the cusp geometry of that
manifold.

\subsection{Definitions and main results}\label{sec:defi-main}

To describe our results precisely, we need some definitions.  As we
will be working with both orientable and non-orientable surfaces, we
need to clarify the notion of an essential surface.

\begin{define}\label{def:essential}
Let $M$ be an orientable $3$--manifold and $S \subset M$ a properly
embedded surface.  We say that $S$ is \emph{essential} in $M$ if the
boundary of a regular neighborhood of $S$, denoted $\widetilde{S}$, is
incompressible and boundary--incompressible.
\end{define}

Note that if $S$ is orientable, then $\widetilde{S}$ consists of two
copies of $S$, and the definition is equivalent to the standard notion
of ``incompressible and boundary--incompressible'' for orientable
surfaces.

\begin{define}\label{def:types}
Let $M$ be a compact $3$--manifold with boundary consisting of tori,
and let $S$ be a properly embedded essential surface in $M$.  An
\emph{accidental parabolic} on $S$ is a free homotopy class of a
closed curve that is not boundary--parallel on $S$ but can be
homotoped to the boundary of $M$.  If $M$ is hyperbolic, then the
embedding of $S$ into $M$ induces a faithful representation $\rho \co
\pi_1(S) \hookrightarrow \pi_1(M) \subset PSL(2, \CC)$.  In this case,
an \emph{accidental parabolic} is a non-peripheral element of
$\pi_1(S)$ that is is mapped by $\rho$ to a parabolic in $\pi_1(M)$.
A surface $S$ with accidental parabolics is called \emph{accidental}.

If $M$ is hyperbolic, the surface $S$ is called \emph{quasifuchsian}
if the embedding $S \hookrightarrow M$ lifts to a topological plane in
${\HH}^3$ whose limit set $ \Lambda \subset \partial {\HH}^3$ is a
topological circle.  Note that we permit $S$ to be non-orientable: in
this case, the two disks bounded by the Jordan curve $\Lambda$ will be
be interchanged by isometries corresponding to $\pi_1(S)$.

Finally, we say the surface $S$ is a \emph{semi-fiber} if it is a
fiber in $M$ or covered by a fiber in a two-fold cover of $M$. If $S$ is a semi-fiber but not a fiber, we call it a \emph{strict} semi-fiber.
\end{define}

By the work of Thurston \cite{thurston:notes} and Bonahon
\cite{bonahonends} (see also Canary, Epstein and Green
\cite{canarynotes}), every properly embedded, essential surface $S$ in
a hyperbolic 3--manifold $M$ falls into exactly one of the three types
in Definition \ref{def:types}: $S$ is either a semi-fiber, or
accidental, or quasifuchsian.

We will apply the above definitions to surfaces constructed from
\emph{Kauffman states} of link diagrams.  For any crossing of a link
diagram $D(K)$, there are two associated diagrams, obtained by
removing the crossing and reconnecting the diagram in one of two ways,
called the \emph{$A$--resolution} and \emph{$B$--resolution} of the
crossing, shown in Figure \ref{fig:splicing}.

\begin{figure}
	\centerline{\input{figures/splicing.pstex_t}}
\caption{$A$-- and $B$--resolutions of a crossing.}
\label{fig:splicing}
\end{figure}
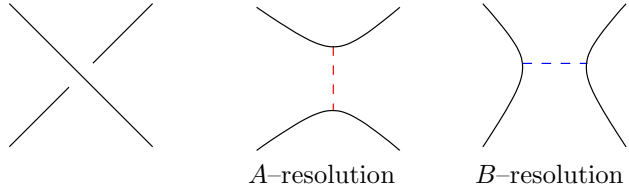

A choice of $A$-- or $B$--resolution for each crossing of $D$ is
called a \emph{Kauffman state} \cite{KaufJones}. The result of
applying a Kauffman state $\sigma$ to a link diagram $D$ is a
collection of circles $s_\sigma$ disjointly embedded in the projection
plane $S^2\subset S^3$.  These circles bound embedded disks whose
interiors can be made disjoint by pushing them below the projection
plane.  Now, at each crossing of $D$, we connect the pair of
neighboring disks by a half--twisted band to construct a \emph{state
  surface} $S_\sigma \subset S^3$ whose boundary is $K$.

State surfaces generalize the classical checkerboard knot surfaces,
and they have recently appeared in the work of several authors,
including Przytycki \cite{przytycki-survey} and Ozawa \cite{ozawa}.
They are the primary object of interest in this paper, for certain
states.  In order to describe these states, we need a few more
definitions.

From the collection of state circles $s_\sigma$ we obtain a trivalent
graph $H_\sigma$ by attaching edges, one for each crossing of the
original diagram $D(K)$, as shown by the dashed lines of Figure
\ref{fig:splicing}.  As in \cite{fkp:gutsjp}, the edges of $H_\sigma$
that come from crossings of the diagram are referred to as
\emph{segments}, and the other edges are portions of state circles.
See Figure \ref{fig:state-sfc-ex}.

\begin{figure}
  \includegraphics{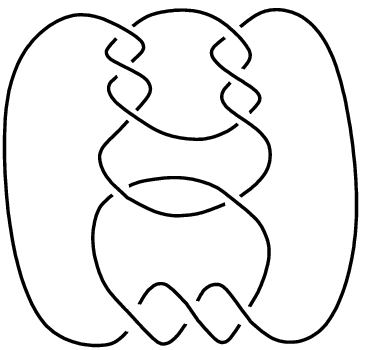} \hspace{.2in}
  \includegraphics{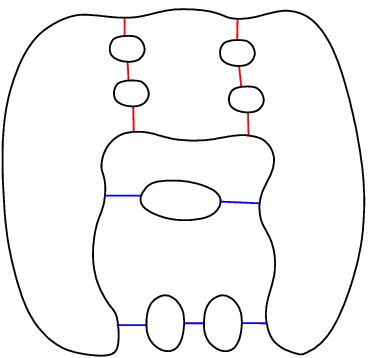} \hspace{.2in}
  \includegraphics{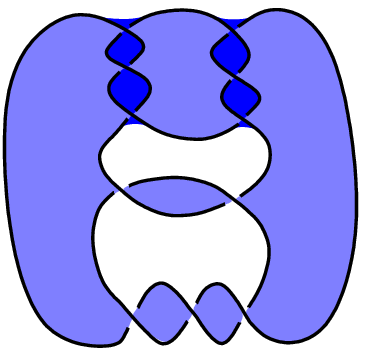}
  \caption{Left: An example link diagram.  Middle: the graph
    $H_\sigma$ corresponding to an adequate, homogeneous state
    $\sigma$. Red edges are $A$--resolutions and blue edges are
    $B$--resolutions.  Right: state surface $S_\sigma$.}
  \label{fig:state-sfc-ex}
\end{figure}

In the literature, a graph that is more common than the graph
$H_\sigma$ is the \emph{state graph} $G_\sigma$, which is formed from
$H_\sigma$ by collapsing components of $s_\sigma$ to vertices.  Remove
redundant edges between vertices to obtain the \emph{reduced state
  graph} $G'_\sigma$.

\begin{define}\label{def:adequate}  
Following Lickorish and Thistlethwaite \cite{lick-thistle,
  thi:adequate}, a state $\sigma$ of a diagram $D$ is said to be
\emph{adequate} if every segment of $H_\sigma$ has its endpoints on
distinct state circles of $s_\sigma$.  In this case, the diagram $D$
is called $\sigma$--adequate.  When $\sigma$ is the all--$A$ state
(all--$B$ state), we call the diagram $A$--adequate ($B$--adequate).

In any state $\sigma$, the circles of $s_\sigma(D)$ divide the
projection plane into components.  Every crossing of $D$ is associated
to a segment of $H_\sigma$, which belongs to one of these components.
Label each segment $A$ or $B$, in accordance with the choice of
resolution at this crossing.  We say that the state $\sigma$ is
\emph{homogeneous} if all edges in a complementary region of
$s_\sigma$ have the same $A$ or $B$ label.  In this case, we say that
$D$ is \emph{$\sigma$--homogeneous}.  An example is shown in Figure
\ref{fig:state-sfc-ex}.  If a link $K$ admits a diagram that is both
$\sigma$--homogeneous and $\sigma$--adequate, for the same state
$\sigma$, we call $K$ \emph{homogeneously adequate}.
\end{define}

Ozawa showed that the state surface $S_\sigma$ of an adequate,
homogeneous state $\sigma$ is essential in the link complement
\cite{ozawa}.  A different proof of this fact follows from machinery
developed by the authors \cite{fkp:gutsjp}.  The state surfaces $S_A$
and $S_B$ corresponding to the all--$A$ and all--$B$ states,
respectively, also play a significant role in quantum topology.  In
\cite{fkp:gutsjp}, we show that coefficients of the colored Jones
polynomials detect topological information about these surfaces.  For
instance, if $K$ is an $A$--adequate link then $S_A$ is a fiber in the
link complement precisely when a particular coefficient vanishes (and
similarly for $S_B$).

In this paper, we show that for hyperbolic link complements, the
colored Jones polynomial completely determines the geometric type of
$S_A$ in the Thurston trichotomy of Definition \ref{def:types}.  To
state our result, let
$$ J^n_K(t)= \alpha_n t^{m_n}+ \beta_n t^{m_n-1}+ \ldots + \beta'_n
t^{r_n+1}+ \alpha'_n t^{r_n},
$$ denote the $n$-th \emph{colored Jones polynomial} of a link $K$,
where $m_n$ and $r_n$ denote the highest and the lowest degree.
Recall that $J^2_K(t)$ is the usual Jones polynomial.  Suppose that
$K$ is a link admitting an $A$--adequate diagram $D$.  Consider the
all--$A$ state graph $\GA$ and the reduced graph $\GRA$.  By
\cite{lick-thistle, thi:adequate, dasbach-lin:head-tail}, for all
$n>1$, we have $\abs{\alpha'_n}=1$ and $ \abs{\beta'_n} = 1 -
\chi(\GRA)$. Thus we may define the \emph{stable coefficient}
\begin{equation}\label{eq:stable-value}
\beta_K'\: := \: \abs{\beta'_n} \: = \: 1 - \chi(\GRA).
\end{equation}
Similarly, if $D$ is $B$--adequate, then $\abs{\alpha_n}=1$ and
$\abs{\beta_n} = 1 - \chi(\GRB)$, hence there is a stable coefficient
$\beta_K := \abs{\beta_n} \: = \: 1 - \chi(\GRB) = 1 - \chi(\GRB)$.

Finally, recall that a link diagram $D$ is called \emph{prime} if any
simple closed curve that meets the diagram transversely in two points
bounds a region of the projection plane without any crossings.  A
prime knot or link admits a prime diagram. 

One of our results is the following theorem.

\begin{theorem}\label{thm:qsfjp}
Let $D(K)$ be a prime, $A$--adequate diagram of a hyperbolic link
$K$. Then the stable coefficient $\beta'_K$ determines the
geometric type of the all--$A$ surface $S_A$, as follows:
\begin{itemize}
\item If $\beta'_K=0$, then $S_A$ is a fiber in $S^3 \setminus K$.
\item If $\beta'_K\neq 0$, then $S_A$ is quasifuchsian.
\end{itemize}
\end{theorem}

Similarly, if $D(K)$ is a prime $B$--adequate diagram of a hyperbolic
link $K$, then the stable coefficient $\beta_K$ determines the
geometric type of $S_B$. This surface will be a fiber if $\beta_K =
0$, and quasifuchsian otherwise.

\begin{remark} 
The class of $A$-- or $B$--adequate links includes all alternating
links, positive and negative closed braids, closed 3--braids,
Montesinos links, Conway sums of alternating tangles and planar cables
of all the above. It also includes all but a handful of prime knots up
to 12 crossings.  See \cite[Section 1.3]{fkp:gutsjp} for more
discussion and references.  The class of homogeneously adequate links
includes all of the above and also contains the homogeneous links
studied by Cromwell \cite{cromwell}.

We note that the class of homogeneously adequate links is strictly
larger than that of $A$-- and $B$--adequate links: For example,
consider the knot $K=12n0873$ of Knotinfo \cite{knotinfo}.  Its Jones
polynomial $
J_{K}(t)=3t^{-4}-7t^{-3}+11t^{-2}-14t^{-1}+15-14t+11t^2-7t^3+3t^4$ is
not monic, hence $K$ is neither $A$-- nor $B$--adequate.  On the other
hand, according to \cite{knotinfo}, $K$ is written as the closure of
the homogeneous braid $ b=\sigma_1 \sigma_2
\sigma_3^{-1}\sigma_4^{-1}\sigma_2 \sigma_3^{-1} \sigma_1\sigma_2
\sigma_3^{-1}\sigma_2\sigma_4^{-1}\sigma_3^{-1},$ where $\sigma_i$
denotes the $i$-th standard generator of the 5--string braid group.
It is not hard to see that the Seifert state of the closed braid
diagram is homogeneous and adequate.

At this writing, it is not known whether \emph{every} hyperbolic link
admits a homogeneously adequate diagram.  See \cite{ozawa} and
\cite[Chapter 10]{fkp:gutsjp} for related discussion and questions.
\end{remark}

The main result of the paper is the following theorem.

\begin{theorem}\label{thm:quasifuch}
Let $D(K)$ be a prime link diagram with an adequate, homogeneous state
$\sigma$.  Then the state surface $S_\sigma$ is essential, and admits
no accidental parabolics.  Furthermore, $S_\sigma$ is a semi-fiber
whenever it is a fiber, which occurs if and only if $G'_\sigma$ is a
tree.
\end{theorem}

Theorem \ref{thm:qsfjp} follows immediately from Theorem
\ref{thm:quasifuch}: simply restrict to $A$--adequate diagrams, and
note that equation \eqref{eq:stable-value} above implies $\beta'_K =
0$ precisely when $\GRA$ is a tree.

The result that checkerboard surfaces in hyperbolic alternating link
complements are quasifuchsian (cf \cite{adams:quasi-fuchsian}) also
follows immediately from Theorem \ref{thm:quasifuch}.  This is because
checkerboard surfaces correspond to the all--$A$ and all--$B$ states
of alternating link complements, which are always homogeneous and
adequate, and the corresponding graphs $G'_A$ and $G'_B$ will be
trees only when the reduced alternating diagram of the link is a
$(2,q)$ torus link, which is not hyperbolic. 
  
The main novel content of Theorem \ref{thm:quasifuch} is that
$S_\sigma$ is never accidental.  Indeed, in \cite[Theorem
  5.21]{fkp:gutsjp}, we showed that $S_\sigma$ is a fiber precisely
when the reduced state graph $G'_\sigma$ is a tree and that it is
never a strict semi-fiber.  Thus, by Thurston and Bonahon
\cite{bonahonends}, for a hyperbolic link $K$ the surface $S_{\sigma}$
is quasifuchsian precisely when $G'_\sigma$ is not a tree.


\subsection{Organization}
In Section \ref{sec:accidental}, we discuss accidental parabolic
elements in the fundamental group of a state surface.  We observe that
the existence of such elements gives rise to an essential embedded
annulus in the complement of the state surface, and then exclude such
annuli in in the case where $K$ is a knot (see Theorem
\ref{thm:noaccidental-knot}).  This, in particular, implies the main
results for knots.

Proving Theorem \ref{thm:quasifuch} in the more general case of links
is harder, and involves knowing more details about the complement of
the state surface.  In Section \ref{sec:polyhedra}, we describe the
structure of an ideal decomposition of the state surface complement,
which was first constructed in \cite{fkp:gutsjp}.  In Section
\ref{sec:links}, we study normal annuli in this polyhedral
decomposition, and prove that such an annulus can never realize an
accidental parabolic.  We expect that some of the combinatorial
results established in Section \ref{sec:links} will also prove useful
for studying more general essential surfaces in the complements of
homogeneously adequate links.

%% file: figures/splicing.pstex_t
\begin{picture}(0,0)%
\includegraphics{figures/splicing.pstex}%
\end{picture}%
\setlength{\unitlength}{3947sp}%
\begingroup\makeatletter\ifx\SetFigFont\undefined%
\gdef\SetFigFont#1#2#3#4#5{%
  \reset@font\fontsize{#1}{#2pt}%
  \fontfamily{#3}\fontseries{#4}\fontshape{#5}%
  \selectfont}%
\fi\endgroup%
\begin{picture}(3898,1192)(1189,-941)
\put(4126,-886){\makebox(0,0)[lb]{\smash{{\SetFigFont{10}{12.0}{\familydefault}{\mddefault}{\updefault}{\color[rgb]{0,0,0}$B$--resolution}%
}}}}
\put(2701,-886){\makebox(0,0)[lb]{\smash{{\SetFigFont{10}{12.0}{\familydefault}{\mddefault}{\updefault}{\color[rgb]{0,0,0}$A$--resolution}%
}}}}
\end{picture}%

%% file: annuli.tex
\section{Embedded annuli and knots}\label{sec:accidental}

In this section, we prove that if an essential state surface
$S_\sigma$ has an accidental parabolic, that is, if a non-peripheral
curve in $S_\sigma$ is homotopic to the boundary, then such a homotopy
can be realized by an embedded annulus. This will quickly lead to a
proof of Theorem \ref{thm:quasifuch} in the special case where $K$ is
a knot.

\begin{define}\label{def:cut}
Let $M$ be a compact orientable $3$--manifold with $\bdy M$ consisting
of tori, and $S \subset M$ a properly embedded surface. We use the
notation $M \cut S$ to denote the path--metric closure of $M \setminus
S$. Up to homeomorphism, $M \cut S$ is the same as the complement of a
regular neighborhood of $S$.

The \emph{parabolic locus} $P$ is the portion of $\bdy M$ that remains
in $\bdy(M \cut S)$. If every torus of $\bdy M$ is cut along $S$, then
the parabolic locus $P$ will consist of annuli.  Otherwise, it will
consist of annuli and tori.  The remaining, non-parabolic boundary
$\bdy (M \cut S) \setminus \bdy M$ can be identified with
$\widetilde{S}$, the boundary of a regular neighborhood of $S$.  In
the special case where $M = S^3 \setminus K$ is a link complement and
$S = S_\sigma$ is a state surface, we use the notation $M_\sigma$ to
refer to $M \cut S_\sigma = (S^3\setminus K)\cut S_\sigma = S^3\cut
S_\sigma$.
\end{define}

The following lemma recounts a standard argument.  It should be
compared, for example, to \cite[Lemma 2.1]{cooper-long}.

\begin{lemma}\label{lemma:annulus}
Let $M$ be a compact orientable $3$--manifold with $\bdy M$ consisting
of tori.  Let $S \subset M$ be a properly embedded essential surface
such that $\bdy S$ meets every component of $\bdy M$.  If $S$ has an
accidental parabolic, then there is an embedded essential annulus $A
\subset M \cut S$ with one boundary component on $\widetilde{S}$ and
the other on the parabolic locus $P = \bdy M \cut \bdy S$.
Furthermore, the component $\bdy A \subset P$ is parallel to a
component of $\bdy \widetilde{S}$.
\end{lemma}

\begin{proof}
If $S$ admits an accidental parabolic, then there exists a
non-peripheral closed curve $\gamma$ on $S$ which is freely homotopic
into $\bdy M$ through $M$.  The free homotopy defines a map of an
annulus $A_1$ into $M$, with one boundary component on $\gamma$ and
the other on $\bdy M$.  Put $A_1$ into general position with respect
to $S$.  Because $S$ may be non-orientable, we will work with the
boundary of a regular neighborhood of $S$, denoted $\widetilde{S}$.
We may move the component of $\partial A_1$ on $\widetilde{S}$ in a
bi-colar of $S$ to be disjoint from $\widetilde{S}$.  Now, any closed
curve of intersection of $A_1$ and $\widetilde{S}$ that bounds a disk
in $A_1$ can be pushed off $\widetilde{S}$ by the fact that
$\widetilde{S}$ is incompressible (because $S$ is essential,
Definition \ref{def:essential}).  Likewise, we can push off any arcs
of intersection of $A_1$ and $\widetilde{S}$ which have both endpoints
on $\bdy M$, because $\widetilde{S}$ is boundary incompressible.
Because we have moved the other boundary component of $A_1$ off of
$\widetilde{S}$, there can be no arcs of intersection of $A_1$ and
$\widetilde{S}$.  There may be closed curves of intersection that are
essential on $A_1$.

Apply a homotopy to minimize the number of closed curves of
intersection.  Then there is a sub-annulus $A_2 \subseteq A_1$ that is
outermost, i.e.\ has one boundary component on $\bdy M$, and one on
$\widetilde{S}$.  Note $A_2$ might equal $A_1$.  By construction, the
interior of $A_2$ is mapped to the interior of $M \cut \widetilde{S}$.
We may assume that the mapping of $A_2$ into $M \cut \widetilde{S}$ is
non-degenerate, i.e.\ cannot be homotoped into the boundary of $(M
\cut \widetilde{S})$, for otherwise the map of $A_1$ into $M$ can be
simplified by homotopy.
Now, the annulus theorem of Jaco \cite[Theorem VIII.13]{Jaco} implies
there exists an essential embedding of an annulus $A$ into $M \cut
\widetilde{S}$, with one end in $\widetilde{S}$ and the other end on
the parabolic locus $P$.

Now $M \cut \widetilde{S}$ is the disjoint union of an $I$--bundle
over $S$ and a manifold homeomorphic to $M \cut S$, with the
non-parabolic portions of $M \cut S$ homeomorphic to the non-parabolic
portions of $M \cut \widetilde{S}$. The $I$--bundle over $S$ cannot
contain any accidental parabolic annuli, for such an annulus would
realize a homotopy between a peripheral and a non-peripheral curve in
$S$. Thus $A$ must lie in the component of $M \cut \widetilde{S}$ which
is homeomorphic to $M \cut S$.
\end{proof}

 In \cite{fkp:gutsjp}, we constructed a
polyhedral decomposition of $M_\sigma$.  In the next section, we will
outline several of its pertinent features, while referring to \cite{fkp:gutsjp, fkp:survey} for details.  To handle the
case where $K$ is a knot, we mainly need the following result.

\begin{theorem}[Theorem 3.23 of \cite{fkp:gutsjp}]
  \label{thm:sigma-homo-poly}
Let $D(K)$ be a connected diagram with an adequate, homogeneous state
$\sigma$. There is a decomposition of $M_\sigma$ into 4--valent,
checkerboard colored ideal polyhedra. The ideal vertices lie on the
parabolic locus $P$, the white faces are glued to other polyhedra, and
the shaded faces lie in $\widetilde{S_\sigma}$, the non-parabolic part
of $\bdy M_\sigma$. 
\end{theorem} 

Normal surface theory  ensures that the intersections of the
annulus $A$ of Lemma \ref{lemma:annulus} with the polyhedral
decomposition of $M_\sigma$ can be taken to have a number of nice
properties.

\begin{define}
\label{def:normal}
We say a surface is in \emph{normal form} if it satisfies the
following conditions:
\begin{enumerate}[(i)]
\item\label{normal1} Each component of its intersection with the
  polyhedra is a disk.
\item\label{normal2} Each disk intersects a boundary edge of a
  polyhedron at most once.
\item\label{normal3} The boundary of such a disk cannot enter and
  leave an ideal vertex through the same face of the polyhedron.
\item\label{normal4} The surface intersects any face of the polyhedra
  in arcs.
\item\label{normal5} No such arc can have endpoints in the same ideal
  vertex of a polyhedron, nor in a vertex and an adjacent edge.
\end{enumerate}
\end{define}

\begin{lemma}\label{lemma:trapezoids}
Let $D(K)$ be a link diagram with an adequate, homogeneous state $\sigma$. Suppose the state surface $S_\sigma$ has an accidental parabolic. Then the embedded annulus $A$ of Lemma \ref{lemma:annulus} can be moved
by isotopy into normal form with respect to the polyhedral decomposition of
$S^3 \cut S_\sigma$.  The intersections of $A$ with white faces of the
polyhedra are all lines running from one boundary component of $A$ to
the other.
\end{lemma}

\begin{proof}
Note that $M_\sigma = S^3 \cut S_\sigma$ is topologically a handlebody, hence irreducible.
 By Haken \cite{haken:normal} we may isotope $A$ into normal form.
Consider the intersections of $A$ with white faces.  A component of 
intersection cannot be a simple closed curve, by item \eqref{normal4} of the definition of normal form. If a component of intersection is an arc with both endpoints on
$N(K)$, we can remove this intersection  by \cite[Lemma
  3.20]{fkp:gutsjp}: every white face of the polyhedral decomposition
is boundary incompressible in $M \cut S_\sigma$.  Similarly, an arc of intersection has both endpoints
on $S_\sigma$, then we may pass to an outermost such arc and obtain a \emph{normal bigon}, that is a normal disk with two sides. This contradicts \cite[Proposition 3.24]{fkp:gutsjp}: the polyhedral decomposition of $M\cut S_\sigma$ contains no normal bigons.
\end{proof}

We are now ready  to prove that an adequate, homogeneous state surface for a
\emph{knot}   admits no
accidental parabolics.  

\begin{theorem}\label{thm:noaccidental-knot}
Let $D(K)$ be a knot diagram with an adequate, homogeneous state $\sigma$.  Then the state surface
$S_\sigma$ cannot be accidental.
\end{theorem}

\begin{proof}
Suppose not: suppose $S_\sigma$ is accidental.  Then Lemma
\ref{lemma:annulus} implies there is an embedded annulus $A$ in
$M_\sigma$ with one boundary component on $\widetilde{S_\sigma}$ and the other on
the parabolic locus $N(K)$.  Consider the intersections of $A$ with a fixed white
face $W$.  Because the boundary component of $A$ on $N(K)$ runs
parallel to $S_\sigma$, the annulus $A$ must intersect each ideal
vertex of $W$.  Moreover, by Lemma \ref{lemma:trapezoids}, any
component of intersection $A \cap W$ runs from the component of $A$ on
$N(K)$ to the component on $\widetilde{S_\sigma}$.  Hence on $W$, this
intersection is an arc from an ideal vertex of $W$ to one of the sides
of $W$ (shaded faces are on $\widetilde{S_\sigma}$).

Because $A$ is normal, item (v) of Definition \ref{def:normal} implies
that such an arc cannot run from an ideal vertex to an adjacent edge.
But now we have a contradiction: there is no way to embed a collection
of arcs in $W$ such that each arc meets one ideal vertex and one side
of $W$ without having an arc that runs from an ideal vertex to an
adjacent edge.
\end{proof}

%% file: poly.tex
\section{Details of the ideal polyhedra}\label{sec:polyhedra}
The proof  of Theorem \ref{thm:noaccidental-knot} for links
requires knowing more information
about the the polyhedral decomposition of \cite{fkp:gutsjp}.  In this
section, we review some of the relevant features, referring to
\cite[Chapters 2--4]{fkp:gutsjp} for more details.
 
  A \emph{non-prime arc} is an arc with both endpoints
on the same state circle of $H_\sigma$, which separates the subgraph
of $H_\sigma$ on one side of the state circle into two graphs which
each contain segments.  Such a subgraph is called a \emph{non-prime
  half--disk}.  A collection of non-prime arcs is called
\emph{maximal} if, once we cut along all such arcs and all state
circles, the graph decomposes into subgraphs each of which contains a
segment, and no larger collection of non-prime arcs has the same
property.

Let $\{\alpha_1, \dots, \alpha_n\}$ denote a maximal collection of
non-prime arcs.  We define a \emph{polyhedral region} to be a
nontrivial region of the complement of the state circles and the
$\alpha_i$.  The manifold $M_\sigma = S^3 \cut S_\sigma$
decomposes into one upper polyhedron and several lower polyhedra.
Each lower polyhedron corresponds to precisely one of these polyhedral
regions. Furthermore, the state circles and segments that meet this
polyhedral region naturally define a subgraph of $H_\sigma$ and a
prime, alternating sub-diagram of $D(K)$. The $1$--skeleton of the
lower polyhedron is exactly the same as the $4$--valent projection
graph of the prime, alternating link diagram corresponding to this
subgraph of $H_\sigma$.

Our maximal collection of non-prime arcs ensures that the polyhedral
regions correspond to prime sub-diagrams of $D(K)$ and to lower
polyhedra without normal bigons.
Meanwhile, the vertices, edges, and faces of the upper polyhedron have
the following description.

\begin{enumerate}
\item Each white face corresponds to a (nontrivial,
  i.e.\ non-innermost disk) complementary region of $H_\sigma \cup
  (\cup_{i=1}^n \alpha_i)$.
\item Each shaded face lies on $\widetilde{S_\sigma}$, and is the
  neighborhood of a tree that we call a \emph{spine}. The spine is
  \emph{directed}, in that each edge has a natural orientation.
  Innermost disks are sources.  Arrows are attached corresponding to
  \emph{tentacles}, which run from a state circle adjacent to a
  segment (the \emph{head}) and then turn left (all--$A$ case) or
  right (all--$B$ case) and have their \emph{tail} along a state
  circle, as well as \emph{non-prime switches}, where four arrows meet
  at a non-prime arc.  See \cite[Figure 3.7]{fkp:gutsjp} for an
  illustration of these terms.

  When an arc is running through the directed spine in the direction
  of the arrows, we say it is running \emph{downstream}.
\item Each vertex of the upper polyhedron corresponds to a strand of
  $D(K)$ between consecutive under-crossings. In the graph $H_\sigma$,
  this strand follows a \emph{zig-zag},
  that is, an alternating sequence of portions of state circles and segments (possibly zero segments).
  See Figure \ref{fig:vertex}, right, for a zig-zag with one segment.
\item Each edge of the upper polyhedron starts at the head of a
  tentacle of a shaded face.  As a result, ideal edges can be given an
  orientation, which matches the orientation of the directed spine in
  that tentacle.
\end{enumerate}

White faces of the lower polyhedra are glued to white faces of the
upper polyhedron. We may transfer combinatorial information about the
upper polyhedron into the lower ones via a map called the clockwise map.

\begin{define}\label{def:clockwise}
Let $W$ be a white face of the upper polyhedron, with $n$ sides.  If
$W$ belongs to an all--$A$ polyhedral region, the \emph{clockwise map}
$\phi$ on $W$ is defined by composing the gluing map of the white face
with a $2\pi/n$ clockwise rotation.  See Figure \ref{fig:clockwise}.  If $W$ belongs to an all--$B$ polyhedral region,
the map $\phi$ is defined by composing the gluing map with a $2\pi/n$
counter-clockwise rotation.  We sometimes call it the
\emph{counter-clockwise map}.
\end{define}

\begin{figure}
\begin{center}
  \input{figures/clockwise-example.pstex_t}
  \end{center}
  \caption{An arc $\beta$ and its image under the gluing
    map and the clockwise map.}
  \label{fig:clockwise}
\end{figure}
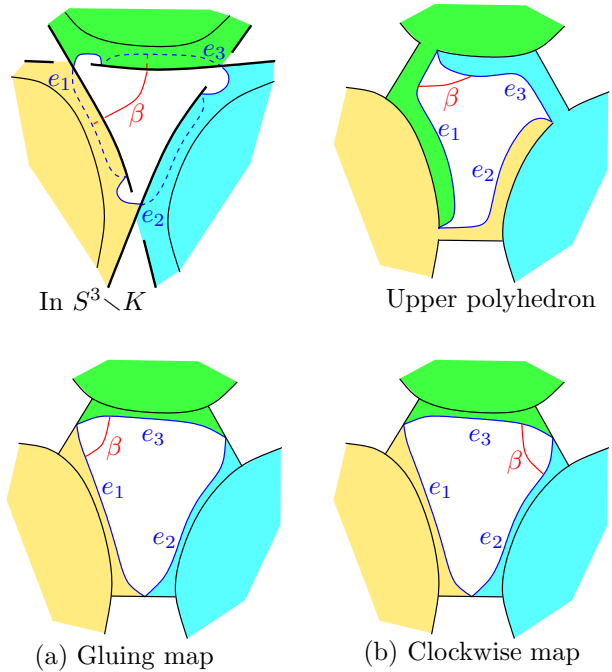

As illustrated in Figure \ref{fig:clockwise}, the clockwise or
counter-clockwise map $\phi$ is orientation--preserving. This is
because the ``viewer'' is in the upper polyhedron: we see the boundary
of the upper polyhedron from the inside, and each lower polyhedron
from the outside. With this convention, the gluing map preserves
orientations, hence $\phi$ does also.

If the special case where $D(K)$ is prime and alternating, there is
exactly one lower polyhedron, and the $1$--skeleta of both the upper
and lower polyhedra coincide with the $4$--valent graph of the
diagram. In this case, both the clockwise and counter-clockwise maps
can be seen as the ``identity map'' on regions of the diagram
\cite{lackenby:volalt}. In the non-alternating setting, more details
about the clockwise map can be found in \cite[Sections 4.2 and
  4.5]{fkp:gutsjp}.

The following lemma describes the effect of the clockwise and
counterclockwise maps on \emph{normal squares}, that is, normal disks
with four sides.  Here we allow a portion of the quadrilateral that
runs over $N(K)$ (i.e. a neighborhood of an ideal vertex of the
polyhedral decomposition) to count as a side.

\begin{lemma}\label{lemma:clockwise}
Let $U$ be a polyhedral region of the projection plane, let $W_1,
\dots, W_k$ be the white faces in $U$, and let $P'$ be the lower
polyhedron associated to $U$.  Then the clockwise (counter-clockwise)
map $\phi\co W_1 \cup \dots \cup W_k \to P'$ has the following
properties:
\begin{enumerate}
\item\label{i:1} If $x$ and $y$ are points on the boundary of white
  faces in $U$ that belong to the same shaded face of the upper
  polyhedron, then $\phi(x)$ and $\phi(y)$ belong to the same shaded
  face of $P'$.
\item\label{i:2} Let $S$ be a normal square in the upper polyhedron
  with two sides on shaded faces (that is, on $\widetilde{S_\sigma}$)
  and two sides on white faces $V$ and $W$, with $V$ and $W$ both
  belonging to polyhedral region $U$.  Let $\beta_v = S\cap V$ and
  $\beta_w = S\cap W$.  Then the arcs $\phi(\beta_v)$ and
  $\phi(\beta_w)$ can be joined along shaded faces to give a normal
  square $S' \subset P'$, defined uniquely up to normal isotopy.
  Write $S' = \phi(S)$.
\item\label{i:3} Let $S$ be a square in the upper polyhedron with one
  side on a shaded face, two sides on white faces $V$ and $W$, and the
  fourth side on $N(K)$, meeting the upper polyhedron in a single
  ideal vertex between $V$ and $W$.  Suppose further that $V$ and $W$
  both belong to polyhedral region $U$.  Then the arcs $\beta_v= S\cap
  V$ and $\beta_w = S\cap W$ meet at a single ideal vertex in the
  lower polyhedron, and their other endpoints can be joined along a
  shaded face to give a normal square $S'\subset P'$, defined uniquely
  up to normal isotopy.  Write $S' = \phi(S)$.
\item\label{i:4} If $S_1$ and $S_2$ are disjoint normal squares in the
  upper polyhedron, all of whose white faces belong to $U$, then
  $\phi(S_1)$ is disjoint from $\phi(S_2)$.
\end{enumerate}
\end{lemma}

\begin{proof}
Items \eqref{i:1} and \eqref{i:2} are proved in \cite[Lemma
  4.8]{fkp:gutsjp} in the case where $U$ is an all--$A$ polyhedral
region.  The proof of the all--$B$ case is identical, with
``clockwise'' replaced by ``counter-clockwise.''  We do need to prove
items \eqref{i:3} and \eqref{i:4}.

For \eqref{i:3}, let $S$ be a normal square in the upper polyhedron as
described: sides $\beta_w$ and $\beta_v$ are arcs in white faces $V$
and $W$ lying in $U$, meeting at a single ideal vertex in the upper
polyhedron.  The proof of \eqref{i:2} implies that the endpoints of
$\phi(\beta_w)$ and $\phi(\beta_v)$ on shaded faces can be connected
by an arc in a single shaded face.  Thus we focus on the endpoints
which lie on an ideal vertex.

Because the clockwise (or counter-clockwise) map takes vertices of
white faces to vertices, each of the arcs $\phi(\beta_w)$ and $\phi(\beta_v)$
still has one end on an ideal vertex in $P'$.  We need to verify
that they have this end on the same ideal vertex of $P'$.

Assume, without loss of generality, that $U$ is an all--$A$ polyhedral
region, and the map $\phi$ is clockwise.  (The proof for the
counter-clockwise map will be identical.)

Recall that an ideal vertex in the upper polyhedron corresponds to a
zig-zag in the graph $H_\sigma$.  Because $V$ and $W$ belong to the
same polyhedral region, they are not separated by any state
circles. As a result, the vertex between them must be a zig-zag with a
single segment. This single segment corresponds to a single
over-crossing of the diagram and a single segment of the graph
$H_\sigma$, as in Figure \ref{fig:vertex}.  But now, the clockwise map
rotates the vertices of each white face clockwise, to lie in the
center of the next segment of $H_\sigma$ in the clockwise direction.
Now, the  endpoints of $\beta_v$ and $\beta_w$ are rotated to the
center of the same segment, namely the segment corresponding to the
single over--crossing of the ideal vertex.

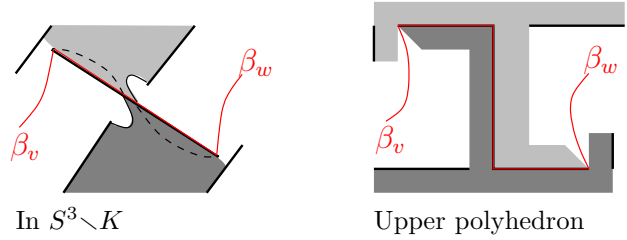
\begin{figure}
  \input{figures/cwise-vertex.pstex_t}
  \caption{Two arcs in white faces in the same all--$A$ polyhedral region,
    meeting the same ideal vertex, must be as shown. In an all--$B$ region, 
    the picture is mirror reversed.}
  \label{fig:vertex}
\end{figure}

Finally, for item \eqref{i:4}, as $\phi$ is a homeomorphism on white
faces, sides of $\phi(S_1)$ and $\phi(S_2)$ on white faces are
disjoint.  If both $\phi(S_1)$ and $\phi(S_2)$ pass through the
interior of a shaded face $F$, then the argument of \cite[Lemma
  4.8]{fkp:gutsjp} shows they are disjoint.  If $\phi(S_1)$ passes
through the interior of a shaded face $F$ and $\phi(S_2)$ passes
through a vertex, then they will be disjoint in $F$.  Finally, if
$\phi(S_1)$ and $\phi(S_2)$ both pass through ideal vertices of $F$,
if they pass through distinct vertices then their images will be
disjoint.  If they pass through (a neighborhood of) the same vertex in
the upper polyhedron, since the squares are disjoint, in the adjacent
white faces the arcs of $S_1$ must lie on the same side of the arc of
$S_2$.  This will be preserved by the clockwise map acting on both
faces, and so the images can be connected at the vertex in a manner
that keeps them both disjoint.
\end{proof}

%% file: figures/clockwise-example.pstex_t
\begin{picture}(0,0)%
\includegraphics{figures/clockwise-example.pstex}%
\end{picture}%
\setlength{\unitlength}{3947sp}%
\begingroup\makeatletter\ifx\SetFigFont\undefined%
\gdef\SetFigFont#1#2#3#4#5{%
  \reset@font\fontsize{#1}{#2pt}%
  \fontfamily{#3}\fontseries{#4}\fontshape{#5}%
  \selectfont}%
\fi\endgroup%
\begin{picture}(3799,4212)(862,-3613)
\put(3985, 19){\makebox(0,0)[lb]{\smash{{\SetFigFont{10}{12.0}{\familydefault}{\mddefault}{\updefault}{\color[rgb]{0,0,1}$e_3$}%
}}}}
\put(1459,-2480){\makebox(0,0)[lb]{\smash{{\SetFigFont{10}{12.0}{\familydefault}{\mddefault}{\updefault}{\color[rgb]{0,0,1}$e_1$}%
}}}}
\put(1779,-2797){\makebox(0,0)[lb]{\smash{{\SetFigFont{10}{12.0}{\familydefault}{\mddefault}{\updefault}{\color[rgb]{0,0,1}$e_2$}%
}}}}
\put(1709,-2148){\makebox(0,0)[lb]{\smash{{\SetFigFont{10}{12.0}{\familydefault}{\mddefault}{\updefault}{\color[rgb]{0,0,1}$e_3$}%
}}}}
\put(3513,-2480){\makebox(0,0)[lb]{\smash{{\SetFigFont{10}{12.0}{\familydefault}{\mddefault}{\updefault}{\color[rgb]{0,0,1}$e_1$}%
}}}}
\put(3833,-2797){\makebox(0,0)[lb]{\smash{{\SetFigFont{10}{12.0}{\familydefault}{\mddefault}{\updefault}{\color[rgb]{0,0,1}$e_2$}%
}}}}
\put(3763,-2148){\makebox(0,0)[lb]{\smash{{\SetFigFont{10}{12.0}{\familydefault}{\mddefault}{\updefault}{\color[rgb]{0,0,1}$e_3$}%
}}}}
\put(1121, 81){\makebox(0,0)[lb]{\smash{{\SetFigFont{10}{12.0}{\familydefault}{\mddefault}{\updefault}{\color[rgb]{0,0,1}$e_1$}%
}}}}
\put(2095,272){\makebox(0,0)[lb]{\smash{{\SetFigFont{10}{12.0}{\familydefault}{\mddefault}{\updefault}{\color[rgb]{0,0,1}$e_3$}%
}}}}
\put(1709,-772){\makebox(0,0)[lb]{\smash{{\SetFigFont{10}{12.0}{\familydefault}{\mddefault}{\updefault}{\color[rgb]{0,0,1}$e_2$}%
}}}}
\put(3781,-490){\makebox(0,0)[lb]{\smash{{\SetFigFont{10}{12.0}{\familydefault}{\mddefault}{\updefault}{\color[rgb]{0,0,1}$e_2$}%
}}}}
\put(3243,-1306){\makebox(0,0)[lb]{\smash{{\SetFigFont{10}{12.0}{\familydefault}{\mddefault}{\updefault}{\color[rgb]{0,0,0}Upper polyhedron}%
}}}}
\put(1037,-3553){\makebox(0,0)[lb]{\smash{{\SetFigFont{10}{12.0}{\familydefault}{\mddefault}{\updefault}{\color[rgb]{0,0,0}(a) Gluing map}%
}}}}
\put(3104,-3525){\makebox(0,0)[lb]{\smash{{\SetFigFont{10}{12.0}{\familydefault}{\mddefault}{\updefault}{\color[rgb]{0,0,0}(b) Clockwise map}%
}}}}
\put(3571,-235){\makebox(0,0)[lb]{\smash{{\SetFigFont{10}{12.0}{\familydefault}{\mddefault}{\updefault}{\color[rgb]{0,0,1}$e_1$}%
}}}}
\put(1068,-1324){\makebox(0,0)[lb]{\smash{{\SetFigFont{10}{12.0}{\familydefault}{\mddefault}{\updefault}{\color[rgb]{0,0,0}In $S^3\setminus K$}%
}}}}
\put(4004,-2325){\makebox(0,0)[lb]{\smash{{\SetFigFont{10}{12.0}{\familydefault}{\mddefault}{\updefault}{\color[rgb]{1,0,0}$\beta$}%
}}}}
\put(1487,-2239){\makebox(0,0)[lb]{\smash{{\SetFigFont{10}{12.0}{\familydefault}{\mddefault}{\updefault}{\color[rgb]{1,0,0}$\beta$}%
}}}}
\put(1629,-144){\makebox(0,0)[lb]{\smash{{\SetFigFont{10}{12.0}{\familydefault}{\mddefault}{\updefault}{\color[rgb]{1,0,0}$\beta$}%
}}}}
\put(3606,-46){\makebox(0,0)[lb]{\smash{{\SetFigFont{10}{12.0}{\familydefault}{\mddefault}{\updefault}{\color[rgb]{1,0,0}$\beta$}%
}}}}
\end{picture}%

%% file: figures/cwise-vertex.pstex_t
\begin{picture}(0,0)%
\includegraphics{figures/cwise-vertex.pstex}%
\end{picture}%
\setlength{\unitlength}{3947sp}%
\begingroup\makeatletter\ifx\SetFigFont\undefined%
\gdef\SetFigFont#1#2#3#4#5{%
  \reset@font\fontsize{#1}{#2pt}%
  \fontfamily{#3}\fontseries{#4}\fontshape{#5}%
  \selectfont}%
\fi\endgroup%
\begin{picture}(3899,1486)(774,-5596)
\put(826,-5536){\makebox(0,0)[lb]{\smash{{\SetFigFont{10}{12.0}{\familydefault}{\mddefault}{\updefault}{\color[rgb]{0,0,0}In $S^3\setminus K$}%
}}}}
\put(3076,-5536){\makebox(0,0)[lb]{\smash{{\SetFigFont{10}{12.0}{\familydefault}{\mddefault}{\updefault}{\color[rgb]{0,0,0}Upper polyhedron}%
}}}}
\put(2223,-4548){\makebox(0,0)[lb]{\smash{{\SetFigFont{12}{14.4}{\familydefault}{\mddefault}{\updefault}{\color[rgb]{1,0,0}$\beta_w$}%
}}}}
\put(789,-5076){\makebox(0,0)[lb]{\smash{{\SetFigFont{12}{14.4}{\familydefault}{\mddefault}{\updefault}{\color[rgb]{1,0,0}$\beta_v$}%
}}}}
\put(4392,-4529){\makebox(0,0)[lb]{\smash{{\SetFigFont{12}{14.4}{\familydefault}{\mddefault}{\updefault}{\color[rgb]{1,0,0}$\beta_w$}%
}}}}
\put(3083,-5030){\makebox(0,0)[lb]{\smash{{\SetFigFont{12}{14.4}{\familydefault}{\mddefault}{\updefault}{\color[rgb]{1,0,0}$\beta_v$}%
}}}}
\end{picture}%

%% file: links.tex
\section{The case of links}\label{sec:links}

The goal of this section is to prove Theorem \ref{thm:noaccidental},
which generalizes Theorem \ref{thm:noaccidental-knot} to links with
multiple components. We note that, unlike Theorem \ref{thm:noaccidental-knot}, this result needs the hypothesis of prime diagrams.

\begin{theorem}\label{thm:noaccidental}
Let $D(K)$ be a prime, $\sigma$--adequate, $\sigma$--homogeneous link
diagram.  Then the state surface $S_\sigma$ has no accidental
parabolics.
\end{theorem}

Suppose, to the contrary, that the state surface $S_\sigma$ is
accidental. Then Lemma \ref{lemma:annulus} implies there is an
embedded annulus $A \subset M_\sigma$ with one boundary component on
$\widetilde{S_\sigma}$ and the other on the parabolic locus
$N(K)$. After placing $A$ in normal form (as in Lemma
\ref{lemma:trapezoids}), we obtain a number of normal squares in
individual polyhedra.  Following the annulus, these squares $A_1,
\dots, A_n$ alternate lying in the upper polyhedron, then a lower
polyhedron, then the upper polyhedron again, and so on.  Each $A_i$
has two sides on white faces, one on a shaded face, and one on $N(K)$.
Finally, each $A_i$ is glued to $A_{i+1}$ along a white face of the
decomposition. Throughout this section, we adopt the convention that
odd-numbered squares are in the upper polyhedron.

The proof of Theorem \ref{thm:noaccidental} is broken up into a number
of lemmas, which analyze the intersection pattern of these squares and
their clockwise images. In \S \ref{sec:reductions}, we perform
the first reductions in the proof and show Proposition
\ref{prop:more4}: the annulus $A$ must be composed of at least $4$
squares, and some white face met by $A$ has at least $4$ sides. Then,
in \S \ref{subsec:annuli-squares}, we use the conclusion of
Proposition \ref{prop:more4} to restrict the possibilities for $D(K)$
further and further, until we show in \S \ref{subsec:completing} that $S_\sigma$ has no accidental
parabolics.

\subsection{First reductions in the proof}\label{sec:reductions}
We begin with  the following lemma.

 \begin{lemma}\label{lemma:more2}
The annulus $A$ must contain at least $4$ normal squares. 
\end{lemma}

\begin{proof}
Since the squares $A_i$ alternate between the upper and lower
polyhedra, the number of these squares must be even. Thus, suppose $A$
consists of only two squares: $A_1$ in the upper polyhedron and $A_2$
in a lower polyhedron.  Since $A_1$ is glued to $A_2$ along both of
its white faces, these white faces $V$ and $W$ must lie in the same
polyhedral region $U$.

By Lemma \ref{lemma:clockwise} \eqref{i:3}, we may map $A_1$ into the
lower polyhedron by a map $\phi$.  The normal square $A'_1 = \phi(A_1)$ runs through
one ideal vertex, white faces $V$ and $W$, and a single shaded
face. Without loss of generality, the map $\phi$ rotates clockwise.

Recall that $A_1$ is glued to $A_2$ across $V$, and that the clockwise
map $\phi$ differs from the gluing map by a $2\pi/n$ rotation. Thus in
$V$, the arc of $A_2$ differs from that of $A'_1$ by a single
clockwise rotation.  Similarly in $W$.  Thus the arcs of $A'_1$ and of
$A_2$ in $V$ and $W$ must be as in Figure \ref{fig:2squares}, left.
The dashed lines in that figure indicate the clockwise motions of
$A_2$.  These must be the lines on the white faces $V$ and $W$
corresponding to $A'_1$.  Note that the points where the dashed lines
meet a vertex, labeled $x$ and $y$, must agree in the polyhedron.
Putting these two points together, the diagram must be as in
Figure \ref{fig:2squares}, right.

\begin{figure}
\begin{center}
  \input{figures/2squares.pstex_t}
  \hspace{.1in}
  \input{figures/2squares-1v.pstex_t}
\end{center}
  \caption{A picture of a lower polyhedron, in the case where $A$ is
    cut into only two squares $A_1$ and $A_2$. 
  }
  \label{fig:2squares}
\end{figure}
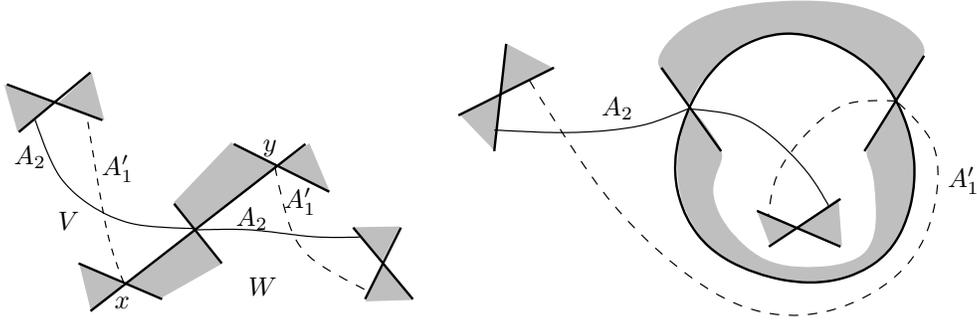

But note in particular that there is a circle coming from the edges of
the polyhedra which separates the two endpoints of the solid line representing $A_2$. (It also separates the two endpoints of the dashed line representing $A'_1$.)  Since these endpoints must be connected by an
embedded arc of $A_2$ in a shaded face, we have a contradiction.
\end{proof}

\begin{lemma}\label{lemma:triangle-diffpoly}
Let $A_1 \subset A$ be a normal square in the upper polyhedron. If
both white faces met by $A_1$ are triangles, these triangles are in
different polyhedral regions.
\end{lemma}

\begin{proof}
Suppose that $A_1$ lies in the upper polyhedron with both of its white
faces in the same polyhedral region, and both of those white faces are
triangles.  Then we may map $A_1$ to the lower polyhedron of this
polyhedral region via the clockwise (or counter-clockwise)
map. Without loss of generality, we may assume that the map $\phi$ is
clockwise in this region. Since $A_1$ is glued to $A_2$ and $A_n$,
this lower polyhedron contains both $A_2$ and $A_n$.

The square $A_2$ in a lower polyhedron runs through one shaded face,
two triangular white faces, and one ideal vertex.  By Lemma
\ref{lemma:clockwise}, part \eqref{i:3}, $A'_1 = \phi(A_1)$ is also a
normal square that passes through an ideal vertex.  Because $A_1$ is
glued to $A_2$, we have one side of $A'_1$ and one side of $A_2$ in
the same white triangle, and these sides differ by a single clockwise
rotation.  Thus $A'_1$ and $A_2$ must be as shown in Figure
\ref{fig:triangle}, left.  Note that the shaded face met by $A_2$ and
the shaded face met by $A'_1$ cannot agree: if they did, this single
shaded face would meet the white face along two edges, contradicting
\cite[Proposition 3.24]{fkp:gutsjp} (No normal bigons).  Hence the
arcs shown in that figure can connect to closed curves only if the
triangular faces labeled $V_1$ and $V_2$ actually coincide.

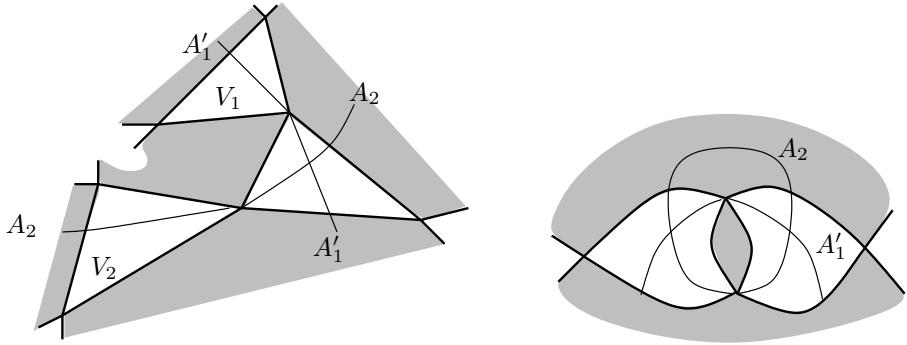
\begin{figure}
  \input{figures/triangle-face1.pstex_t}
  \hspace{.3in}
  \input{figures/triangle-face2.pstex_t}
  \caption{Configurations for triangular faces in the same polyhedral
    region.}
  \label{fig:triangle}
\end{figure}

Since $V_1 = V_2$, the configuration must be as in Figure
\ref{fig:triangle}, right.  But now, recall that $A_1$ is glued to
square $A_n$ along this white face $V_1 = V_2$.  By Lemma
\ref{lemma:more2}, the squares $A_2$ and $A_n$ are distinct.
Furthermore, since $A$ is embedded, $A_2$ and $A_n$ are disjoint.
However, the side of $A_n$ on the face $V_1=V_2$ differs from $A'_1$
by a single clockwise rotation.  It is impossible for this arc to be
disjoint from $A_2$, which is a contradiction.
\end{proof}

We can now prove the main result of this section.

\begin{prop}\label{prop:more4}
The annulus $A$ consists of at least $4$ squares. In addition, some
white face met by $A$ has at least $4$ sides.
\end{prop}

\begin{proof}
The first claim in the proposition is proved in Lemma
\ref{lemma:more2}.  To prove the second claim, let $A_1 \subset A$ be
a normal square in the upper polyhedron.  We will show that this
particular normal square meets a white face with at least $4$ sides.

First we rule out white faces that are bigons.  In a bigon face, each
edge is adjacent to each of the two vertices.  Thus any arc from an
ideal vertex to an edge would violate condition \eqref{normal5} of
Definition \ref{def:normal}, meaning $A$ cannot be normal if it meets
a bigon face.  This contradiction implies every white face met by
$A_1$ has at least $3$ sides.

If both white faces met by $A_1$ are triangles, then Lemma
\ref{lemma:triangle-diffpoly} implies these triangles are in different
polyhedral regions.  To study this situation, we need the following
lemma.

\begin{lemma}\label{lemma:triangV}
Suppose $A_1$ is a normal square in the upper polyhedron, with one side on an ideal vertex, two sides on white faces $V$ and $W$, where
$V$ is triangular, and one side, labeled $\gamma$, on a shaded face.
Label the state circles around $V$ so that $\bdy A_i$ runs from a
vertex of $V$ on the state circle $C_1$ to a tentacle whose tail is on
the state circle $C_2$.  Then either   
\begin{enumerate}
\item $W$ is inside the region $R_1$ on
the opposite side of $C_1$ from $V$; or
\item $W$ is inside $R_2$ on the opposite
side of $C_2$ from $V$.  
\end{enumerate}
Furthermore, when we direct $\gamma$ from $V$ to $W$, it
runs across $C_1$ or $C_2$, respectively, running downstream.  See
Figure \ref{fig:triangV}.
\end{lemma}

\begin{proof}
The square $A_1$ has one side on the parabolic locus, which is a
vertex of the upper polyhedron.  Each vertex is a zig-zag.  Because
$A_1$ meets a vertex on $C_1$, part of the zig-zag must lie on $C_1$.

If all of the zig-zag lies on $C_1$, that is if the zig-zag consists of a
single bit of state surface, then $W$ lies in $R_1$ on the opposite
side of $C_1$ from $V$.

If the zig-zag contains one or more segments, then at least one
segment of the zig-zag is attached to $C_1$, on one side or the other.
If the segment is attached to $C_1$ on the side of the region $R_1$,
then $W$ must be inside $R_1$. (Otherwise, there would be a staircase
from state circle $C_1$ back to $C_1$, contradicting the Escher Stairs
Lemma \cite[Lemma 3.4]{fkp:gutsjp}.)  If the segment is attached to
the side opposite $R_1$, then because it belongs to a single vertex,
it must in fact be the segment labeled $s$ in Figure
\ref{fig:triangV}, which connects $C_1$ to $C_2$ alongside face $V$.
In this case, the zig-zag includes a portion of $C_2$, and $W$ will
lie on one side or the other of $C_2$. By the assumption that $V$ and
$W$ are in different polyhedral regions, $W$ must lie inside the
region $R_2$ on the opposite side of $C_2$ from $V$.

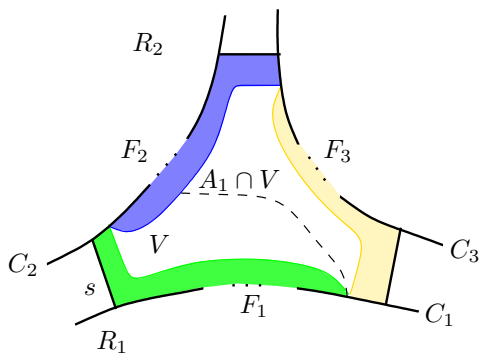
\begin{figure}
\input{figures/triangV.pstex_t}
\caption{Notation for Lemma \ref{lemma:triangV}. The conclusion of the lemma is that square $A_1$ must run through shaded face $F_2$ to a shaded face $W$ contained in region $R_1$ or region $R_2$.}
\label{fig:triangV}
\end{figure}

Now we argue that $\gamma$ runs downstream across $C_1$ or $C_2$, when
directed away from $V$ towards $W$.  As in Figure \ref{fig:triangV}, the shaded face containing $\gamma$ is called $F_2$. For ease of exposition, we also refer to $F_2$ as the blue face.  Thus $\gamma$ starts next to white face $V$ by entering a blue tentacle adjacent to $C_2$.

First suppose $W$ is in $R_2$.  If $\gamma$ crosses $C_2$ immediately
from the tail of the blue tentacle, then it must do so running
downstream, since only heads of tentacles (rather than no non-prime
switches or innermost disks) can attach to tails of
tentacles on the opposite side of a state circle.  So suppose $\gamma$ runs upstream into
the head of the blue tentacle, crossing state circle $C_3$.  Since $C_3$ does not separate $V$ and $W$, in fact
$\gamma$ must cross it twice, and the Utility Lemma \cite[Lemma
  3.11]{fkp:gutsjp} implies that $\gamma$ crosses it first running
upstream, then downstream.  Between the second time $\gamma$
crosses $C_3$ and the first time it crosses $C_2$, $\gamma$ must exit
out of every non-prime half--disk it enters, else such a disk would
separate $C_2$ and $C_3$.  But no half--disk can separate $C_2$ and $C_3$, because they are connected by a
segment. Thus the Downstream Lemma \cite[Lemma
  3.10]{fkp:gutsjp} implies $\gamma$ crosses $C_2$ running downstream.

Finally, suppose $W$ is in $R_1$.  The arc $\gamma$ begins in a blue
tentacle with head on $C_3$ and tail on $C_2$.  If $\gamma$ crosses
$C_2$ first, it will be running downstream.  But $C_2$ does not
separate $V$ and $W$ in this case, so $\gamma$ must cross it twice.
This contradicts the Utility Lemma.  Thus $\gamma$ crosses $C_3$
first, running upstream.  Again it crosses $C_3$ twice, and by the
Utility Lemma, the second crossing of $C_3$ occurs running downstream.  Then, as in the
previous paragraph, the Downstream Lemma implies that $\gamma$ crosses
$C_1$ running downstream. 
\end{proof}

Now we finish the proof of Proposition \ref{prop:more4}.

Let the notation be as in Lemma \ref{lemma:triangV}. In addition, as Figure \ref{fig:triangV}, let 
$F_i$ be the shaded face that has a tentacle lying on state circle $C_i$. Thus $\gamma$ runs through shaded face $F_2$.

To finish the proof, we pull a side of $A_1$ off the parabolic locus, i.e.\ off the ideal vertex, and into shaded face $F_1$ or $F_3$. This creates a normal square with two white sides and two shaded sides.

If $W$ is in $R_1$, pull $A_1$ off the ideal vertex and into the  tentacle of $F_1$,
to obtain an arc $\sigma \subset F_1$.  This arc $\sigma$ must run
downstream across $C_1$, by the Utility Lemma\cite[Lemma
  3.11]{fkp:gutsjp} and Downstream Lemma \cite[Lemma
  3.10]{fkp:gutsjp} 
 (as in the above argument).

If $W$ is in $R_2$, pull $A_1$ off the ideal vertex and into the  tentacle of $F_3$,
obtaining an arc $\sigma \subset F_3$.  Again the arc $\sigma$
must run downstream across $C_2$.

In either case, we have arcs $\gamma$ and $\sigma$ which run
downstream from the same state circle (either $C_1$ if $W \subset R_1$, or $C_2$ if $W \subset R_2$).  They terminate in the same white
face, namely $W$.  This contradicts the Parallel Stairs Lemma
\cite[Lemma 3.14]{fkp:gutsjp}.
\end{proof}

\subsection{Annuli and squares}\label{subsec:annuli-squares}
In the next sequence of lemmas, we use Proposition \ref{prop:more4} to
set up the proof that the state surface $S_\sigma$ has no accidental
parabolics. The overall theme of the proof is that each successive
lemma places stiffer and stiffer restrictions on the annulus $A$, the
polyhedral decomposition, and the diagram $D(K)$. In the end, we will
reach a contradiction.

So far, we have an essential annulus $A \subset M_\sigma$, composed of
normal squares $A_1, \dots, A_n$. Each of these squares has two sides
on white faces, one on a shaded face, and the final side on an ideal
vertex.

In the arguments below, it is actually easier to view the pieces of
$A$ as squares with two sides on shaded faces and two sides on white
faces.  This is accomplished as follows.  Recall that the parabolic
locus $\bdy N(K) \cut S_\sigma$ consists of annuli. One of the
boundary circles of $A$ is embedded on one of these parabolic annuli.
We may isotope $A$ slightly through $M_\sigma$, to move the boundary
circle of $A$ from the parabolic locus and onto
$\widetilde{S_\sigma}$.

In the polyhedral decomposition, the pushed-off copy of $A$ will be
cut into a collection of normal squares with two sides on white faces
and two sides on shaded faces, such that one side on a shaded face
cuts off a single ideal vertex. We denote these squares by $S_1,
\dots, S_n$.  Note each $S_i$ is obtained by pulling $A_i$ off an
ideal vertex and into an adjacent shaded face.

In fact, there are two different directions in which we may pull $A$
off the parabolic locus. We make the choice as follows.

\begin{convention}\label{conv:pull-off}
Let $V$ be a white face with four or more vertices, which meets the
annulus $A$.  (The existence of such a white face is guaranteed by
Proposition \ref{prop:more4}.)  We arrange the labeling of normal
squares $A_i$ so that square $A_1$ in the upper polyhedron is glued
along $V$ to square $A_2$ in some lower polyhedron.

The normal square $A_1$ meets a vertex of $V$, which means that one
component of $V \setminus A_1$ has two or more vertices.  We pull $A$
off the parabolic locus in the direction of this (larger) component of
$V \setminus A_1$. Thus, if $S_1$ is the normal square corresponding
to $A_1$, the arc $S_1 \cap V$ has at least two vertices on each side.
\end{convention}

\begin{lemma}\label{lemma:same-white-faces}
The annulus $A$ intersects only two white faces, $V$ and $W$, which
belong to the same polyhedral region.  Furthermore, every normal
square $S_i$ intersects $V$ and $W$ in a way that cuts off at least
two vertices on each side.
\end{lemma}

\begin{proof}
Let $V$ be the white face of Convention \ref{conv:pull-off}, and let
$A_1$ and $S_1$ be the corresponding normal squares.  Let $W$ be the
other white face met by $S_1$. Since $S_1$ does not cut off an ideal
vertex in $V$, and is glued to square $S_2$ across $V$,
\cite[Proposition 4.13]{fkp:gutsjp} implies that $V$ and $W$ are in
the same polyhedral region $U$.\protect\footnote{In the monograph
  \cite{fkp:gutsjp}, Proposition 4.13 and Lemma 4.10 are stated for
  $A$--adequate diagrams. As \cite[Section 4.5]{fkp:gutsjp} explains,
  these results and the other structural results about the polyhedra
  also apply to $\sigma$--adequate, $\sigma$--homogeneous diagrams.}

Now, Lemma \ref{lemma:clockwise} part \eqref{i:2} says that we may map
$S_1$ into the lower polyhedron corresponding to $U$ and obtain a
normal square $S'_1 = \phi(S_1)$. Note that the arc $S_1' \cap V$ will
differ from $S_2 \cap V$ by a single rotation, by the definition of
the clockwise (or counter-clockwise) map.  Since $S_1'$ cuts off more
than a single vertex in $V$, \cite[Lemma 4.10]{fkp:gutsjp} implies
that $S_1'$ intersects $S_2$ nontrivially, in both of its white faces.
But this means that $S_2$ meets both $V$ and $W$ in arcs that cut off
more than a single vertex on each side.

The square $S_2$ is glued along $W$ to a square $S_3$ in the upper
polyhedron. The arc $S_3 \cap W$ cuts off more than a single vertex on
each side, because it is glued to $S_2$. Thus, as above,
\cite[Proposition 4.13]{fkp:gutsjp} implies that both white faces of
$S_3$ are in the same polyhedral region $U$, and \cite[Lemma
  4.10]{fkp:gutsjp} implies that $S_3' = \phi(S_3)$ intersects $S_2$
nontrivially, in both of its white faces. In other words, $S_3$ meets
the same white faces $V$ and $W$, in arcs that cut off more than a
single vertex on each side. Continue in this fashion to obtain the
same conclusion for every $S_i$.
\end{proof}

Let $S_i$ be an even-numbered square in a lower polyhedron. Lemma
\ref{lemma:same-white-faces} tells us that $S_i$ is glued to $S_{i-1}$
across $V$ and to $S_{i+1}$ across $W$, where $V$ and $W$ are the same
as $i$ varies.

\begin{define}\label{def:ti}
To continue studying the intersection patterns of normal squares
in the lower polyhedron, we define
$$
T_i =
\begin{cases}
\phi(S_i) & \mbox{if } i \mbox{ is odd} \\
S_i & \mbox{if } i \mbox{ is even.} \\
\end{cases}
$$
Note that every $T_i$ lives in the lower polyhedron of the polyhedral
region $U$.

For every square $T_i$, we label its four sides as follows. The sides
of $T_i$ in white faces $V$ and $W$ are denoted $v_i$ and $w_i$,
respectively.  One shaded side of $S_i$ was created by pulling a side
of $A_i$ off the parabolic locus; the corresponding side of $T_i$ is
denoted $p_i$. (Note that by Lemma \ref{lemma:clockwise}, part
\eqref{i:3}, if an odd-numbered square $S_i$ in the upper polyhedron
has a shaded side that cuts off an ideal vertex, then so does $T_i =
\phi(S_i)$.) We will orient the arcs $v_i$ and $w_i$ so that they
point toward $p_i$, and orient $p_i$ from $v_i$ toward $w_i$. That is,
$p_i$ is oriented from $V$ to $W$.

Similarly, an odd-numbered square $S_i$ in the upper polyhedron also
contains an arc $q_i$ that was pulled off the parabolic locus. As
before, we orient $q_i$ from $V$ to $W$.
\end{define}

\begin{lemma}\label{lemma:orientations}
Let $i$ be even, so that $S_i = T_i$ is in a lower polyhedron, and
suppose that we pulled $S_i$ off an ideal vertex that lies to the
right of $p_i$. Then
\begin{enumerate}
\item\label{orient1} $v_{i-1} = \phi(v_i)$ and $w_{i+1} = \phi(w_i)$,
  with orientations preserved.
\item\label{orient2} $p_{i \pm 1}$ cuts off an ideal vertex to its
  right.
\item\label{orient3} In the upper polyhedron, $q_{i \pm 1}$ also cuts
  off an ideal vertex to its right.
\end{enumerate}
\end{lemma}

\begin{proof}
By construction, $v_i \subset S_i$ is glued to an arc of $S_{i-1} \cap
V$, whose image under $\phi$ is $v_{i-1}$. Similarly for $w_i$ and
$w_{i+1}$. Since $\phi$ is orientation--preserving, \eqref{orient1}
follows.

Conclusion \eqref{orient2} follows immediately from Lemma
\ref{lemma:clockwise} part \eqref{i:3} because $S_i$ was created by
pulling $A_i$ off an ideal vertex in a direction that is consistent
for all $i$. Similarly, conclusion \eqref{orient3} follows from Lemma
\ref{lemma:clockwise} part \eqref{i:3} because $\phi$ is
orientation--preserving.
\end{proof}

\begin{lemma}\label{lemma:encircle-bigons}
Each square $T_i$
encircles a bigon shaded face of the lower polyhedron.
\end{lemma}

\begin{proof}
Assume without loss of generality that $V$ and $W$ are in an all--$A$
polyhedral region.  We may also assume without loss of generality that
$p_2$ was created by pulling $A_2$ off an ideal vertex so that the
vertex lies to the right of $p_2$.  (Otherwise, interchange the labels
of faces $V$ and $W$, reversing the order of the indices and the
orientation on every $p_i$.)

By Lemma \ref{lemma:orientations}, the arc $v_1$ is clockwise from
$v_2$ in face $V$, and $w_3$ is clockwise from $w_2$ in face $W$.
Moreover, $v_2$ intersects both $v_1$ and $v_3$, and similarly $w_2$
intersects both $w_1$ and $w_3$.  But $T_1 = \phi(S_1)$ and $T_3 =
\phi(S_3)$ are clockwise images of disjoint squares, hence are
disjoint by Lemma \ref{lemma:clockwise} \eqref{i:4}.  Thus $T_1$,
$T_2$, and $T_3$ must be as shown in Figure \ref{fig:marcfig12}.  In
particular, $p_1$ and $T_3$ run parallel through the same shaded face.
Dotted lines in the figure indicate that the boundary of the
corresponding shaded face may meet additional vertices. 


\begin{figure}
  \input{figures/squares-1.pstex_t}
  \caption{Proof of Lemma \ref{lemma:encircle-bigons}: squares $T_1$, $T_2$, and $T_3$ must meet a lower
    polyhedron as shown.}
  \label{fig:marcfig12}
\end{figure}
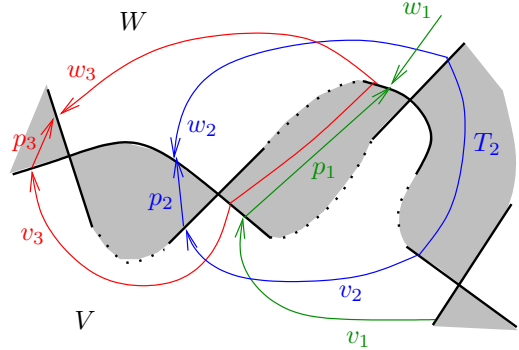

The arc $p_2$ cuts off an ideal vertex to its right, so by Lemma
\ref{lemma:orientations}, the arcs $p_1$ and $p_3$ also cut off ideal
vertices to their right.  Thus the dotted line to the right of $p_1$
in Figure \ref{fig:marcfig12} must actually be solid.  By primeness of
the lower polyhedron, all other dotted lines must also be solid.  Thus
both $T_2$ and $T_3$ each encircle a single bigon shaded face.

We may repeat the above argument with $T_{2k}$ taking the place of
$T_2$, for any $k$, hence each $T_i$ encircles a bigon.
\end{proof}

\begin{lemma}\label{lemma:2q-link}
The white faces $V$ and $W$ met by annulus $A$ are the only white
faces of the polyhedral decomposition.  As a consequence, $D(K)$ is
the standard diagram of a $(2,n)$ torus link, and $S_\sigma$ is an
annulus.
\end{lemma}

\begin{proof}
Recall that by Lemma \ref{lemma:same-white-faces}, there is a
polyhedral region $U$ containing white faces $V$ and $W$, such that
every normal square $S_i$ passes through $V$ and $W$. These normal
squares define squares $T_i$ in the lower polyhedron, as in Definition
\ref{def:ti}. By Lemma \ref{lemma:encircle-bigons}, every $T_i$
encircles a bigon shaded face of this lower polyhedron. The number of
these bigons is $n$, the same as the number of normal squares in $A$.

This is enough to conclude that all the shaded faces of the lower
polyhedron corresponding to $U$ are bigons, chained end to end.  Thus
$V$ and $W$ are the only white faces of this lower polyhedron.  The
$1$--skeleton of this lower polyhedron coincides with the standard
diagram of a $(2,n)$ torus link, as on the left of Figure
\ref{fig:bigon-preimage}.

If the diagram $D(K)$ is prime and alternating, there is only one
lower polyhedron, whose $1$--skeleton corresponds to $D(K)$. Thus
$D(K)$ is the standard diagram of a $(2,n)$ torus link, where $n$ is
even. The rest of the argument reduces us to this case.

In the general case, the upper polyhedron may be more complicated.
However, one polyhedral region in the upper polyhedron looks like that
of a $(2,n)$ torus link, as in the middle panel of Figure
\ref{fig:bigon-preimage}.  A priori, there may be additional segments
attached to the opposite sides of all state circles involved.  This is
indicated in that figure by the dashed lines along state circles.

For each square $T_i$ in the lower polyhedron, label three sides of
$T_i$ by $v_i$, $w_i$, and $p_i$, as in Lemma
\ref{lemma:orientations}.  Focusing attention on $T_2 = S_2$, we may
assume that arc $p_2$ in a shaded face was pulled off an ideal vertex
to its right. (Otherwise, as in Lemma \ref{lemma:encircle-bigons},
switch the labels of $V$ and $W$.)  Applying Lemma
\ref{lemma:orientations} part \eqref{orient2} inductively, we conclude
that for each even index $j$, arc $p_j$ was pulled off an ideal vertex
to its right.

Now, let $i$ be an odd index, so that $S_i$ is a square in the upper
polyhedron. Since $T_i$ encircles an ideal bigon, as in Figure
\ref{fig:bigon-preimage},  the clockwise preimage $S_i =
\phi^{-1}(T_i)$ must be as in the middle panel of Figure
\ref{fig:bigon-preimage}.  By Lemma \ref{lemma:orientations}, the arc
$q_i$ of $S_i$ that was pulled off the parabolic locus must cut off an
ideal vertex to its right.  This means that portions of state circles
adjacent to $q_i$ to its right must actually be solid, to form a
single zig-zag, with no segments to break it up.  In other words, we
have the third panel of Figure \ref{fig:bigon-preimage}.
The third panel of Figure \ref{fig:bigon-preimage} shows two dotted
closed curves, each meeting  the  link diagram  exactly
twice.  Using the hypothesis that the diagram is prime, each of these
closed curves cannot enclose segments (which would correspond to
crossings of the diagram).  

\begin{figure}
\input{figures/bigon-lower.pstex_t}
\hspace{.2in}
\input{figures/bigon-upper1.pstex_t}
\hspace{.2in}
\includegraphics{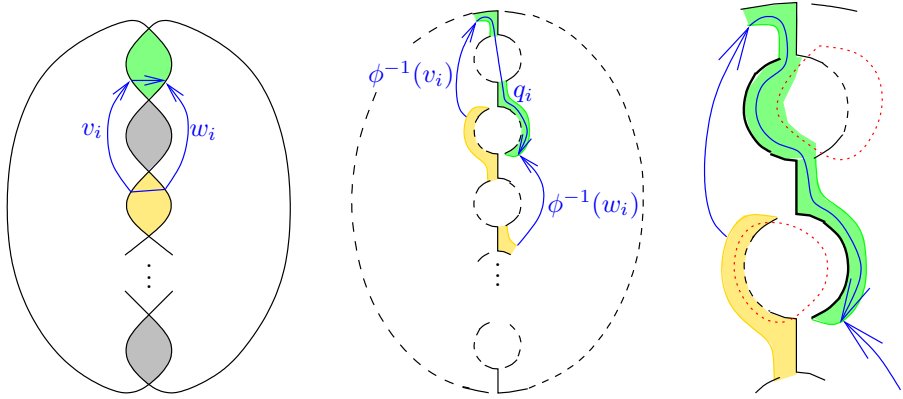}
\caption{Proof of Lemma \ref{lemma:2q-link}. Left:
$T_i$ encircles a bigon in the lower polyhedron.  Center:
The clockwise preimage  $S_i =
\phi^{-1}(T_i)$ in the upper polyhedron. Right: 
The ideal vertex cut off by $S_i$.}
\label{fig:bigon-preimage}
\end{figure}

We conclude that two consecutive state circles in $H_\sigma$ are
innermost, and contain no additional polyhedral regions. Repeating the
same argument for the next odd-numbered square $S_{i+2}$ leads to the
conclusion that the next two state circles in $H_\sigma$ are also
innermost. Continuing in this way, we conclude that there is only one
polyhedral region, which corresponds to the diagram of a $(2,n)$ torus
link.
\end{proof}

\subsection{Completing the proofs}\label{subsec:completing}

We are now ready to prove Theorem \ref{thm:noaccidental} and Theorem
\ref{thm:quasifuch}.

\begin{proof}[Proof of Theorem \ref{thm:noaccidental}]
Suppose that $S_\sigma$ has an accidental parabolic.  Then Lemma
\ref{lemma:annulus} implies there is an embedded essential annulus $A
\subset S^3 \cut S_\sigma$.  By Lemma \ref{lemma:same-white-faces},
$A$ intersects only two white faces, $V$ and $W$.  By Lemma
\ref{lemma:2q-link}, $V$ and $W$ are the only faces of the polyhedral
decomposition, hence $D(K)$ is the standard diagram of a $(2,n)$ torus
link and $S_\sigma$ is an annulus.

Note that the only non-trivial simple closed curve in an annulus is
boundary--parallel. Therefore, the component of $\bdy A$ that lies on
$\widetilde{S_\sigma}$ is actually parallel to $\bdy
\widetilde{S_\sigma}$. This contradicts the assumption that $A$ is an
essential annulus realizing an accidental parabolic.
\end{proof}


\begin{proof}[Proof of Theorem \ref{thm:quasifuch}] 
By \cite[Theorem 3.25]{fkp:gutsjp}, $S_{\sigma}$ is essential in $S^3
\setminus K$, and by Theorem \ref{thm:noaccidental} it has no
accidental parabolics.  By \cite[Theorem 5.21]{fkp:gutsjp} (or
\cite{futer:fiber}) $S_{\sigma}$ is a fiber in $S^3 \setminus K$ if
and only if $G_\sigma'$ is a tree.  Furthermore, by \cite[Theorem
  5.21]{fkp:gutsjp}, if $S_\sigma$ lifts to a fiber in a double cover
of $S^3 \setminus K$, then $M_\sigma$ is an $I$--bundle, hence
$G_\sigma'$ is a tree.

It follows that if $K$ is hyperbolic, the surface $S_\sigma$ is
quasifuchsian if and only if the reduced state graph $G'_\sigma$ is
not a tree.
\end{proof}

%% file: figures/2squares.pstex_t
\begin{picture}(0,0)%
\includegraphics{figures/2squares.pstex}%
\end{picture}%
\setlength{\unitlength}{3947sp}%
\begingroup\makeatletter\ifx\SetFigFont\undefined%
\gdef\SetFigFont#1#2#3#4#5{%
  \reset@font\fontsize{#1}{#2pt}%
  \fontfamily{#3}\fontseries{#4}\fontshape{#5}%
  \selectfont}%
\fi\endgroup%
\begin{picture}(2594,1561)(129,-1150)
\put(1758,-109){\makebox(0,0)[lb]{\smash{{\SetFigFont{10}{12.0}{\familydefault}{\mddefault}{\updefault}{\color[rgb]{0,0,0}$y$}%
}}}}
\put(196,-189){\makebox(0,0)[lb]{\smash{{\SetFigFont{10}{12.0}{\familydefault}{\mddefault}{\updefault}{\color[rgb]{0,0,0}$A_2$}%
}}}}
\put(747,-255){\makebox(0,0)[lb]{\smash{{\SetFigFont{10}{12.0}{\familydefault}{\mddefault}{\updefault}{\color[rgb]{0,0,0}$A'_1$}%
}}}}
\put(826,-1095){\makebox(0,0)[lb]{\smash{{\SetFigFont{10}{12.0}{\familydefault}{\mddefault}{\updefault}{\color[rgb]{0,0,0}$x$}%
}}}}
\put(1590,-575){\makebox(0,0)[lb]{\smash{{\SetFigFont{10}{12.0}{\familydefault}{\mddefault}{\updefault}{\color[rgb]{0,0,0}$A_2$}%
}}}}
\put(1892,-463){\makebox(0,0)[lb]{\smash{{\SetFigFont{10}{12.0}{\familydefault}{\mddefault}{\updefault}{\color[rgb]{0,0,0}$A'_1$}%
}}}}
\put(471,-597){\makebox(0,0)[lb]{\smash{{\SetFigFont{10}{12.0}{\familydefault}{\mddefault}{\updefault}{\color[rgb]{0,0,0}$V$}%
}}}}
\put(1666,-1014){\makebox(0,0)[lb]{\smash{{\SetFigFont{10}{12.0}{\familydefault}{\mddefault}{\updefault}{\color[rgb]{0,0,0}$W$}%
}}}}
\end{picture}%

%% file: figures/2squares-1v.pstex_t
\begin{picture}(0,0)%
\includegraphics{figures/2squares-1v.pstex}%
\end{picture}%
\setlength{\unitlength}{3947sp}%
\begingroup\makeatletter\ifx\SetFigFont\undefined%
\gdef\SetFigFont#1#2#3#4#5{%
  \reset@font\fontsize{#1}{#2pt}%
  \fontfamily{#3}\fontseries{#4}\fontshape{#5}%
  \selectfont}%
\fi\endgroup%
\begin{picture}(3112,1991)(204,-1705)
\put(1126,-436){\makebox(0,0)[lb]{\smash{{\SetFigFont{10}{12.0}{\familydefault}{\mddefault}{\updefault}{\color[rgb]{0,0,0}$A_2$}%
}}}}
\put(3301,-886){\makebox(0,0)[lb]{\smash{{\SetFigFont{10}{12.0}{\familydefault}{\mddefault}{\updefault}{\color[rgb]{0,0,0}$A_1'$}%
}}}}
\end{picture}%

%% file: figures/triangle-face1.pstex_t
\begin{picture}(0,0)%
\includegraphics{figures/triangle-face1.pstex}%
\end{picture}%
\setlength{\unitlength}{3947sp}%
\begingroup\makeatletter\ifx\SetFigFont\undefined%
\gdef\SetFigFont#1#2#3#4#5{%
  \reset@font\fontsize{#1}{#2pt}%
  \fontfamily{#3}\fontseries{#4}\fontshape{#5}%
  \selectfont}%
\fi\endgroup%
\begin{picture}(2940,2144)(158,-1508)
\put(705,-1082){\makebox(0,0)[lb]{\smash{{\SetFigFont{10}{12.0}{\familydefault}{\mddefault}{\updefault}{\color[rgb]{0,0,0}$V_2$}%
}}}}
\put(1485,-16){\makebox(0,0)[lb]{\smash{{\SetFigFont{10}{12.0}{\familydefault}{\mddefault}{\updefault}{\color[rgb]{0,0,0}$V_1$}%
}}}}
\put(173,-820){\makebox(0,0)[lb]{\smash{{\SetFigFont{10}{12.0}{\familydefault}{\mddefault}{\updefault}{\color[rgb]{0,0,0}$A_2$}%
}}}}
\put(2101,-961){\makebox(0,0)[lb]{\smash{{\SetFigFont{10}{12.0}{\familydefault}{\mddefault}{\updefault}{\color[rgb]{0,0,0}$A'_1$}%
}}}}
\put(2326, 14){\makebox(0,0)[lb]{\smash{{\SetFigFont{10}{12.0}{\familydefault}{\mddefault}{\updefault}{\color[rgb]{0,0,0}$A_2$}%
}}}}
\put(1276,314){\makebox(0,0)[lb]{\smash{{\SetFigFont{10}{12.0}{\familydefault}{\mddefault}{\updefault}{\color[rgb]{0,0,0}$A'_1$}%
}}}}
\end{picture}%

%% file: figures/triangle-face2.pstex_t
\begin{picture}(0,0)%
\includegraphics{figures/triangle-face2.pstex}%
\end{picture}%
\setlength{\unitlength}{3947sp}%
\begingroup\makeatletter\ifx\SetFigFont\undefined%
\gdef\SetFigFont#1#2#3#4#5{%
  \reset@font\fontsize{#1}{#2pt}%
  \fontfamily{#3}\fontseries{#4}\fontshape{#5}%
  \selectfont}%
\fi\endgroup%
\begin{picture}(2260,1439)(463,-1269)
\put(2140,-703){\makebox(0,0)[lb]{\smash{{\SetFigFont{10}{12.0}{\familydefault}{\mddefault}{\updefault}{\color[rgb]{0,0,0}$A'_1$}%
}}}}
\put(1902,-98){\makebox(0,0)[lb]{\smash{{\SetFigFont{10}{12.0}{\familydefault}{\mddefault}{\updefault}{\color[rgb]{0,0,0}$A_2$}%
}}}}
\end{picture}%

%% file: figures/triangV.pstex_t
\begin{picture}(0,0)%
\includegraphics{figures/triangV.pstex}%
\end{picture}%
\setlength{\unitlength}{3947sp}%
\begingroup\makeatletter\ifx\SetFigFont\undefined%
\gdef\SetFigFont#1#2#3#4#5{%
  \reset@font\fontsize{#1}{#2pt}%
  \fontfamily{#3}\fontseries{#4}\fontshape{#5}%
  \selectfont}%
\fi\endgroup%
\begin{picture}(2806,2222)(404,-1700)
\put(1303,-1032){\makebox(0,0)[lb]{\smash{{\SetFigFont{10}{12.0}{\familydefault}{\mddefault}{\updefault}{\color[rgb]{0,0,0}$V$}%
}}}}
\put(1614,-621){\makebox(0,0)[lb]{\smash{{\SetFigFont{10}{12.0}{\familydefault}{\mddefault}{\updefault}{\color[rgb]{0,0,0}$A_1 \cap V$}%
}}}}
\put(419,-1151){\makebox(0,0)[lb]{\smash{{\SetFigFont{10}{12.0}{\familydefault}{\mddefault}{\updefault}{\color[rgb]{0,0,0}$C_2$}%
}}}}
\put(3036,-1468){\makebox(0,0)[lb]{\smash{{\SetFigFont{10}{12.0}{\familydefault}{\mddefault}{\updefault}{\color[rgb]{0,0,0}$C_1$}%
}}}}
\put(3195,-1063){\makebox(0,0)[lb]{\smash{{\SetFigFont{10}{12.0}{\familydefault}{\mddefault}{\updefault}{\color[rgb]{0,0,0}$C_3$}%
}}}}
\put(895,-1295){\makebox(0,0)[lb]{\smash{{\SetFigFont{10}{12.0}{\familydefault}{\mddefault}{\updefault}{\color[rgb]{0,0,0}$s$}%
}}}}
\put(1201,239){\makebox(0,0)[lb]{\smash{{\SetFigFont{10}{12.0}{\familydefault}{\mddefault}{\updefault}{\color[rgb]{0,0,0}$R_2$}%
}}}}
\put(976,-1636){\makebox(0,0)[lb]{\smash{{\SetFigFont{10}{12.0}{\familydefault}{\mddefault}{\updefault}{\color[rgb]{0,0,0}$R_1$}%
}}}}
\put(1876,-1411){\makebox(0,0)[lb]{\smash{{\SetFigFont{10}{12.0}{\familydefault}{\mddefault}{\updefault}{\color[rgb]{0,0,0}$F_1$}%
}}}}
\put(1126,-436){\makebox(0,0)[lb]{\smash{{\SetFigFont{10}{12.0}{\familydefault}{\mddefault}{\updefault}{\color[rgb]{0,0,0}$F_2$}%
}}}}
\put(2401,-436){\makebox(0,0)[lb]{\smash{{\SetFigFont{10}{12.0}{\familydefault}{\mddefault}{\updefault}{\color[rgb]{0,0,0}$F_3$}%
}}}}
\end{picture}%

%% file: figures/squares-1.pstex_t
\begin{picture}(0,0)%
\includegraphics{figures/squares-1.pstex}%
\end{picture}%
\setlength{\unitlength}{3947sp}%
\begingroup\makeatletter\ifx\SetFigFont\undefined%
\gdef\SetFigFont#1#2#3#4#5{%
  \reset@font\fontsize{#1}{#2pt}%
  \fontfamily{#3}\fontseries{#4}\fontshape{#5}%
  \selectfont}%
\fi\endgroup%
\begin{picture}(3210,2277)(199,-1649)
\put(888,381){\makebox(0,0)[lb]{\smash{{\SetFigFont{10}{12.0}{\familydefault}{\mddefault}{\updefault}{\color[rgb]{0,0,0}$W$}%
}}}}
\put(608,-1499){\makebox(0,0)[lb]{\smash{{\SetFigFont{10}{12.0}{\familydefault}{\mddefault}{\updefault}{\color[rgb]{0,0,0}$V$}%
}}}}
\put(2668,481){\makebox(0,0)[lb]{\smash{{\SetFigFont{10}{12.0}{\familydefault}{\mddefault}{\updefault}{\color[rgb]{0,.56,0}$w_1$}%
}}}}
\put(2101,-526){\makebox(0,0)[lb]{\smash{{\SetFigFont{10}{12.0}{\familydefault}{\mddefault}{\updefault}{\color[rgb]{0,.56,0}$p_1$}%
}}}}
\put(561, 88){\makebox(0,0)[lb]{\smash{{\SetFigFont{10}{12.0}{\familydefault}{\mddefault}{\updefault}{\color[rgb]{1,0,0}$w_3$}%
}}}}
\put(2268,-1300){\makebox(0,0)[lb]{\smash{{\SetFigFont{10}{12.0}{\familydefault}{\mddefault}{\updefault}{\color[rgb]{0,0,1}$v_2$}%
}}}}
\put(1074,-719){\makebox(0,0)[lb]{\smash{{\SetFigFont{10}{12.0}{\familydefault}{\mddefault}{\updefault}{\color[rgb]{0,0,1}$p_2$}%
}}}}
\put(1314,-259){\makebox(0,0)[lb]{\smash{{\SetFigFont{10}{12.0}{\familydefault}{\mddefault}{\updefault}{\color[rgb]{0,0,1}$w_2$}%
}}}}
\put(3114,-379){\makebox(0,0)[lb]{\smash{{\SetFigFont{10}{12.0}{\familydefault}{\mddefault}{\updefault}{\color[rgb]{0,0,1}$T_2$}%
}}}}
\put(221,-346){\makebox(0,0)[lb]{\smash{{\SetFigFont{10}{12.0}{\familydefault}{\mddefault}{\updefault}{\color[rgb]{1,0,0}$p_3$}%
}}}}
\put(268,-939){\makebox(0,0)[lb]{\smash{{\SetFigFont{10}{12.0}{\familydefault}{\mddefault}{\updefault}{\color[rgb]{1,0,0}$v_3$}%
}}}}
\put(2314,-1585){\makebox(0,0)[lb]{\smash{{\SetFigFont{10}{12.0}{\familydefault}{\mddefault}{\updefault}{\color[rgb]{0,.56,0}$v_1$}%
}}}}
\end{picture}%

%% file: figures/bigon-lower.pstex_t
\begin{picture}(0,0)%
\includegraphics{figures/bigon-lower.pstex}%
\end{picture}%
\setlength{\unitlength}{3947sp}%
\begingroup\makeatletter\ifx\SetFigFont\undefined%
\gdef\SetFigFont#1#2#3#4#5{%
  \reset@font\fontsize{#1}{#2pt}%
  \fontfamily{#3}\fontseries{#4}\fontshape{#5}%
  \selectfont}%
\fi\endgroup%
\begin{picture}(1802,2365)(599,-2139)
\put(1479,-1468){\makebox(0,0)[lb]{\smash{{\SetFigFont{10}{12.0}{\familydefault}{\mddefault}{\updefault}{\color[rgb]{0,0,0}.}%
}}}}
\put(1479,-1338){\makebox(0,0)[lb]{\smash{{\SetFigFont{10}{12.0}{\familydefault}{\mddefault}{\updefault}{\color[rgb]{0,0,0}.}%
}}}}
\put(1479,-1409){\makebox(0,0)[lb]{\smash{{\SetFigFont{10}{12.0}{\familydefault}{\mddefault}{\updefault}{\color[rgb]{0,0,0}.}%
}}}}
\put(1077,-504){\makebox(0,0)[lb]{\smash{{\SetFigFont{10}{12.0}{\familydefault}{\mddefault}{\updefault}{\color[rgb]{0,0,1}$v_i$}%
}}}}
\put(1754,-512){\makebox(0,0)[lb]{\smash{{\SetFigFont{10}{12.0}{\familydefault}{\mddefault}{\updefault}{\color[rgb]{0,0,1}$w_i$}%
}}}}
\end{picture}%

%% file: figures/bigon-upper1.pstex_t
\begin{picture}(0,0)%
\includegraphics{figures/bigon-upper1.pstex}%
\end{picture}%
\setlength{\unitlength}{3947sp}%
\begingroup\makeatletter\ifx\SetFigFont\undefined%
\gdef\SetFigFont#1#2#3#4#5{%
  \reset@font\fontsize{#1}{#2pt}%
  \fontfamily{#3}\fontseries{#4}\fontshape{#5}%
  \selectfont}%
\fi\endgroup%
\begin{picture}(1855,2427)(422,-1875)
\put(1327,-1194){\makebox(0,0)[lb]{\smash{{\SetFigFont{10}{12.0}{\familydefault}{\mddefault}{\updefault}{\color[rgb]{0,0,0}.}%
}}}}
\put(1327,-1135){\makebox(0,0)[lb]{\smash{{\SetFigFont{10}{12.0}{\familydefault}{\mddefault}{\updefault}{\color[rgb]{0,0,0}.}%
}}}}
\put(1327,-1064){\makebox(0,0)[lb]{\smash{{\SetFigFont{10}{12.0}{\familydefault}{\mddefault}{\updefault}{\color[rgb]{0,0,0}.}%
}}}}
\put(1663,-723){\makebox(0,0)[lb]{\smash{{\SetFigFont{10}{12.0}{\familydefault}{\mddefault}{\updefault}{\color[rgb]{0,0,1}$\phi^{-1}(w_i)$}%
}}}}
\put(536, 89){\makebox(0,0)[lb]{\smash{{\SetFigFont{10}{12.0}{\familydefault}{\mddefault}{\updefault}{\color[rgb]{0,0,1}$\phi^{-1}(v_i)$}%
}}}}
\put(1447, -1){\makebox(0,0)[lb]{\smash{{\SetFigFont{10}{12.0}{\familydefault}{\mddefault}{\updefault}{\color[rgb]{0,0,1}$q_i$}%
}}}}
\end{picture}%

%% file: quasifuchsian.bbl
\providecommand{\bysame}{\leavevmode\hbox to3em{\hrulefill}\thinspace}
\providecommand{\href}[2]{#2}
\begin{thebibliography}{10}

\bibitem{adams:quasi-fuchsian}
Colin~C. Adams, \emph{Noncompact {F}uchsian and quasi-{F}uchsian surfaces in
  hyperbolic 3--manifolds}, Alebr. Geom. Topol. \textbf{7} (2007), 565--582.

\bibitem{bonahonends}
Francis Bonahon, \emph{Bouts des vari\'et\'es hyperboliques de dimension
  {$3$}}, Ann. of Math. (2) \textbf{124} (1986), no.~1, 71--158.

\bibitem{canarynotes}
Richard~D. Canary, David B.~A. Epstein, and Paul Green, \emph{Notes on notes of
  {T}hurston}, Analytical and geometric aspects of hyperbolic space
  ({C}oventry/{D}urham, 1984), London Math. Soc. Lecture Note Ser., vol. 111,
  Cambridge Univ. Press, Cambridge, 1987, pp.~3--92.

\bibitem{knotinfo}
Jae~Choon Cha and Charles Livingston, \emph{Knotinfo: Table of knot
  invariants}, 2011, http://www.indiana.edu/$~$knotinfo.

\bibitem{cooper-long}
Daryl Cooper and Darren~D. Long, \emph{Some surface subgroups survive surgery},
  Geom. Topol. \textbf{5} (2001), 347--367 (electronic).

\bibitem{cromwell}
Peter~R. Cromwell, \emph{Homogeneous links}, J. London Math. Soc. (2)
  \textbf{39} (1989), no.~3, 535--552.

\bibitem{dasbach-futer...}
Oliver~T. Dasbach, David Futer, Efstratia Kalfagianni, Xiao-Song Lin, and
  Neal~W. Stoltzfus, \emph{The {J}ones polynomial and graphs on surfaces},
  Journal of Combinatorial Theory Ser. B \textbf{98} (2008), no.~2, 384--399.

\bibitem{dasbach-lin:head-tail}
Oliver~T. Dasbach and Xiao-Song Lin, \emph{On the head and the tail of the
  colored {J}ones polynomial}, Compositio Math. \textbf{142} (2006), no.~5,
  1332--1342.

\bibitem{fenley:qf-seifert}
S{\'e}rgio~R. Fenley, \emph{Quasi-{F}uchsian {S}eifert surfaces}, Math. Z.
  \textbf{228} (1998), no.~2, 221--227.

\bibitem{futer:fiber}
David Futer, \emph{{Fiber detection for state surfaces}}, 2012,
  \mbox{arXiv:1201.1643}.

\bibitem{fkp:survey}
David Futer, Efstratia Kalfagianni, and Jessica~S. Purcell, \emph{Jones
  polynomials, volume, and essential knot surfaces: a survey},
  \mbox{arXiv:1110.6388}, Proceedings of Knots in Poland III, Banach Center
  Publications, to appear.

\bibitem{fkp:gutsjp}
\bysame, \emph{Guts of surfaces and the colored {J}ones polynomial}, Research
  Monograph, Lecture Notes in Mathematics, Vol. 2069, to appear,
  \mbox{arXiv:1108.3370}.

\bibitem{haken:normal}
Wolfgang Haken, \emph{Theorie der {N}ormalfl\"achen}, Acta Math. \textbf{105}
  (1961), 245--375.

\bibitem{Jaco}
William Jaco, \emph{Lectures on three-manifold topology}, CBMS Regional
  Conference Series in Mathematics, vol.~43, American Mathematical Society,
  Providence, R.I., 1980.

\bibitem{KaufJones}
Louis~H. Kauffman, \emph{State models and the {J}ones polynomial}, Topology
  \textbf{26} (1987), no.~3, 395--407.

\bibitem{lackenby:volalt}
Marc Lackenby, \emph{The volume of hyperbolic alternating link complements},
  Proc. London Math. Soc. (3) \textbf{88} (2004), no.~1, 204--224, With an
  appendix by Ian Agol and Dylan Thurston.

\bibitem{lick-thistle}
W.~B.~Raymond Lickorish and Morwen~B. Thistlethwaite, \emph{Some links with
  nontrivial polynomials and their crossing-numbers}, Comment. Math. Helv.
  \textbf{63} (1988), no.~4, 527--539.

\bibitem{MastersZhang}
Joseph~D. Masters and Xingru Zhang, \emph{Closed quasi-{F}uchsian surfaces in
  hyperbolic knot complements}, Geom. Topol. \textbf{12} (2008), no.~4,
  2095--2171.

\bibitem{MenascoReid}
William Menasco and Alan~W. Reid, \emph{Totally geodesic surfaces in hyperbolic
  link complements}, Topology '90 ({C}olumbus, {OH}, 1990), Ohio State Univ.
  Math. Res. Inst. Publ., vol.~1, de Gruyter, Berlin, 1992, pp.~215--226.

\bibitem{ozawa}
Makoto Ozawa, \emph{Essential state surfaces for knots and links}, J. Aust.
  Math. Soc. \textbf{91} (2011), no.~3, 391--404.

\bibitem{przytycki-survey}
J{\'o}zef~H. Przytycki, \emph{From {G}oeritz matrices to quasi-alternating
  links}, The mathematics of knots, Contrib. Math. Comput. Sci., vol.~1,
  Springer, Heidelberg, 2011, pp.~257--316.

\bibitem{thist-tsviet}
Morwen Thistlethwaite and Anastasiia Tsvietkova, \emph{An alternateive approach
  to hyperbolic structures on link complements}, \mbox{arXiv:1108.0510}.

\bibitem{thi:adequate}
Morwen~B. Thistlethwaite, \emph{On the {K}auffman polynomial of an adequate
  link}, Invent. Math. \textbf{93} (1988), no.~2, 285--296.

\bibitem{thurston:notes}
William~P. Thurston, \emph{The geometry and topology of three-manifolds},
  Princeton Univ. Math. Dept. Notes, 1979.

\bibitem{tsutsumi}
Yukihiro Tsutsumi, \emph{Hyperbolic knots spanning accidental {S}eifert
  surfaces of arbitrarily high genus}, Math. Z. \textbf{246} (2004), no.~1-2,
  167--175.

\bibitem{tsvietthesis}
Anastasiia Tsvietkova, \emph{Hyperbolic structures from link diagrams}, Ph.D.
  thesis, University of Tennessee, 2012.

\end{thebibliography}
